\documentclass[12pt, leqno]{amsart}  

\usepackage{setspace}
\usepackage{graphicx}
\usepackage{amssymb,amsmath,amsthm,amscd,accents}
\usepackage{stackengine,scalerel}
\usepackage{mathrsfs}
\usepackage{lscape}
\usepackage{enumerate}
\usepackage{upgreek}
\usepackage{stmaryrd}
\usepackage[usenames,dvipsnames]{color}
\usepackage[colorlinks=true, pdfstartview=FitV,
 linkcolor=blue,citecolor=blue,urlcolor=blue]{hyperref}
\usepackage{tikz}
\tikzset{dynkdot/.style={circle,draw,scale=.38}}
\usetikzlibrary{arrows.meta,arrows}
\usepackage{bm}
\usepackage{marginnote}
\usepackage{verbatim}
\usepackage[normalem]{ulem}  
\usepackage{xspace}

\usepackage[all]{xy}

\usepackage{kotex}
\allowdisplaybreaks[3]

\newcommand{\nc}{\newcommand}
\numberwithin{equation}{section}
\usepackage[colorinlistoftodos]{todonotes}

\newenvironment{red}{\relax\color{red}}{\relax}
\newenvironment{blue}{\relax\color{blue}}{\hspace*{.5ex}\relax}
\newenvironment{jaune}{\relax\marginnote{\ber \scalebox{.6}{\sc{To be deleted}} \er}  \color{Orchid}}{\hspace*{.5ex}\relax}

\newcommand{\ber}{\begin{red}}
\newcommand{\er}{\end{red}}
\nc{\bej}{\begin{jaune}}
\nc{\ej}{\end{jaune}}
\newcommand{\beb}{\begin{blue}}
\newcommand{\eb}{\end{blue}}

\newcommand{\berE}[1]{\begin{red}{}\marginnote{\fbox{\scshape\lowercase{E}}} #1} 
\newcommand{\berMH}[1]{\begin{red}{}\marginnote{\fbox{\scshape\lowercase{MH}}} #1}  
\newcommand{\berJ}[1]{\begin{red}{}\marginnote{\fbox{\scshape\lowercase{J}}} #1}  

\nc{\hs}{\hspace*}
\nc{\ms}{\mspace}
\nc{\st}[1]{\{{#1}\}}

\nc{\qR}[1]{\ttq_{\mspace{-2mu}\raisebox{-.8ex}{${\scriptstyle{#1}}$}}}

\theoremstyle{plain}
\newtheorem{lemma}{Lemma}[section]
\newtheorem{proposition}[lemma]{Proposition}
\newtheorem{theorem}[lemma]{Theorem}

\newtheorem{corollary}[lemma]{Corollary}

\theoremstyle{definition}
\newtheorem{remark}[lemma]{Remark}
\newtheorem{example}[lemma]{Example}

\newtheorem{definition}[lemma]{Definition}

\nc{\ledot}{\mathrel{\le\ms{-11mu}\raisebox{.2ex}{$\cdot$}}}
\nc{\predot}{\mathrel{\preceq\ms{-9mu}\raisebox{.35ex}{$\centerdot$}}}

\newcommand{\Zq}{{\Z[q^{\pm 1/2}]}}

\renewcommand{\le}{\leqslant}
\renewcommand{\ge}{\geqslant}
\renewcommand{\preceq}{\preccurlyeq}

\newcommand{\head}{{\operatorname{hd}}}

\newcommand{\seteq}{\mathbin{:=}}

\newcommand{\tens}{\mathop\otimes}

\newcommand{\g}{\mathfrak{g}}

\newcommand{\R}{\mathbb{R}}
\newcommand{\C}{\mathbb{C}}
\newcommand{\Q}{\mathbb{Q}}
\newcommand{\Z}{\mathbb{Z}\ms{1mu}}

\newcommand{\al}{{\ms{1mu}\alpha}}
\newcommand{\ep}{\epsilon}
\newcommand{\la}{\lambda}

\newcommand{\ch}{{\rm ch}}

\newcommand{\fin}{{{\rm fin}}}

\newcommand{\ta}{\widetilde{a}}
\newcommand{\tb}{\widetilde{b}}

\newcommand{\hd}{\mathrm{hd}}

\newcommand{\frakC}{\mathfrak{C}}

\newcommand{\frakF}{\mathfrak{F}}

\newcommand{\frakH}{\mathfrak{H}}

\newcommand{\frakc}{\mathfrak{c}}

\newcommand{\tfrakc}{\widetilde{\frakc}}

\newcommand{\sfB}{\mathsf{B}}
\newcommand{\sfC}{\mathsf{C}}
\newcommand{\sfD}{\mathsf{D}}

\newcommand{\sfG}{\mathsf{G}}

\newcommand{\sfK}{\mathsf{K}}

\newcommand{\sfP}{\mathsf{P}}

\newcommand{\sfT}{\mathsf{T}}

\newcommand{\sfW}{\mathsf{W}}

\newcommand{\sfg}{\mathsf{g}}

\newcommand{\bbA}{\mathbb{A}}

\newcommand{\bbD}{\mathbb{D}}

\newcommand{\obbA}{\bbA}
\newcommand{\tobbA}{\widetilde{\obbA}}

\newcommand{\bPhi}{{\boldsymbol{\Phi}}}

\newcommand{\bfA}{{ \Z[q^{\pm1/2}]  }}

\newcommand{\bfL}{\mathbf{L}}

\newcommand{\bfR}{\mathbf{R}}

\newcommand{\bfX}{\mathbf{X}}

\newcommand{\bfh}{\mathbf{h}}

\newcommand{\bfj}{\mathbf{j}}

\newcommand{\bfs}{\mathbf{s}}

\newcommand{\calA}{\mathcal{A}}
\newcommand{\calB}{\mathcal{B}}

\newcommand{\calE}{\mathcal{E}}

\newcommand{\calH}{\mathcal{H}}

\newcommand{\calL}{\mathcal{L}}

\newcommand{\calQ}{\mathcal{Q}}
\newcommand{\calR}{\mathcal{R}}

\newcommand{\calU}{\mathcal{U}}

\newcommand{\hcalA}{\widehat{\calA}}

\newcommand{\scrA}{\mathscr{A}}

\newcommand{\scrC}{\mathscr{C}}
\newcommand{\scrD}{\mathscr{D}}

\newcommand{\ttB}{\mathtt{B}}

\newcommand{\ttb}{\mathtt{b}}

\newcommand{\ttq}{\mathtt{q}}

\nc{\bg}{\sigma}

\newcommand{\rmF}{\mathrm{F}}

\newcommand{\rmL}{\mathrm{L}}

\newcommand{\rmR}{\mathrm{R}}

\newlength{\mylength}
\newcommand*{\para}{%
  \rlap{\rotatebox{-30}{\rule[.05ex]{.4pt}{.77em}}}%
  \kern.04em%
  \rlap{\kern.36em\raisebox{0.649519052835em}{\rule{.6em}{.4pt}}}%
  \rule{.6em}{.4pt}\kern-.04em%
  \rotatebox{-30}{\rule[.05ex]{.4pt}{.77em}}}

\newcommand{\wl}{\sfP}

\newcommand{\weyl}{\sfW}
\newcommand{\lan}{\langle}
\newcommand{\ran}{\rangle}

\newcommand{\qtq}[1][{and}]{\quad\text{{#1}}\quad}

\newcommand{\ee}{\end{enumerate}}
\newcommand{\bitem}{\begin{itemize}}
\newcommand{\eitem}{\end{itemize}}
\newcommand{\ben}{\begin{enumerate}[{\rm (1)}]}
\newcommand{\bnum}{\begin{enumerate}[{\rm (i)}]}
\newcommand{\bnump}{\begin{enumerate}[{\rm (i)$'$}]}
\newcommand{\bna}{\begin{enumerate}[{\rm (a)}]}
\newcommand{\bnA}{\begin{enumerate}[{\rm (A)}]}
\newcommand{\bc}{\begin{cases}}
\newcommand{\ec}{\end{cases}}
\newcommand{\ba}{\begin{array}}
\newcommand{\ea}{\end{array}}

\nc{\eqs}[1]{\underset{\raisebox{.4ex}[.7ex][0ex]{$\scriptstyle{#1}$}}{=}}

\newcommand{\Dcan}{{\bbD_\can}}

\newcommand{\rdual}{\scrD}

\newcommand{\ang}[1]{\langle#1\rangle}

\newcommand{\bone}{\mathbf{1}}

\newcommand{\Cg}{{\scrC_\g}}

\newcommand{\Cgz}{{\scrC^0_\g}}

\newcommand{\longepito}[1][]{\xymatrix@C=4ex{{}\ar@{->>}[r]^{#1}&{}}}

\newcommand{\bl}{\bigl(}
\newcommand{\br}{\bigr)}

\newcommand{\uii}{{\boldsymbol{i}}}

\nc{\ake}[1][2ex]{\rule[-.5ex]{0ex}{#1}}

\newcommand{\ujj}{{\boldsymbol{j}}}
\newcommand{\TT}{\mathrm{T}}
\newcommand{\FF}{\mathsf{p}}
\newcommand{\bFF}{ \overline{\FF}}

\newcommand{\Seq}{\mathrm{Seq}}

\newcommand{\kk}{{\Q(q^{1/2})}}

\nc{\col}{\colon}
\nc{\ord}{\mathrm{ord}}
\nc{\catQ}{\scrC_\calQ}
\nc{\catD}{\scrC_\bbD}
\nc{\vs}{\vspace*}
\nc{\D}{\mathscr{D}}

\nc{\tLa}{\widetilde{\Lambda}}
\nc{\ro}{{\rm(}}
\nc{\rf}{{\rm)}\xspace}
\nc{\Aut}{\mathrm{Aut}}
\nc{\can}{\mathrm{can}}
\nc{\Dc}{{\bbD_{\can}}}
\nc{\Cgo}{\scrC_{\mathfrak{g}}^{0}}
\nc{\bb}{\mathtt{b}}
\nc{\catC}{\scrC}
\nc{\braid}{\ttB}
\nc{\Ass}[1][\bbD]{\calE_{#1}}
\nc{\Ci}{C^\uii}
\nc{\Cj}{C^\ujj}
\nc{\condi}[1][K]{with $#1\cap\st{0,1}\not=\emptyset$\xspace}
\nc{\Di}[1][{\uii}]{{\bbD,{#1}}}
\nc{\Vi}[1][\uii]{V^{#1}}
\nc{\Pii}[1][\uii]{P^{#1}}
\nc{\Vdi}[1][{\Di}]{V^{#1}}
\nc{\Pdi}[1][{\Di}]{P^{#1}} 
\nc{\PBix}[1][{[l,r]}]{\Z_{\ge0}^{\oplus{#1}}}
\nc{\mcf}[1][{$[a,b]$}]{maximal commuting family of $i$-boxes in {#1}\xspace}
\nc{\nn}{\nonumber}
\nc{\bwr}{\mbox{\large$\wr$}}
\newcommand{\tch}{\widetilde{\ch}}
\nc{\tchDcan}{\ms{2mu}\tch_{\mspace{.1mu}\raisebox{-.4ex}{${\scriptstyle{\Dcan}}$}}}
\nc{\monoto}[1][]{\xymatrix@C=2ex{\ar@{>->}[r]^-{{#1}}&}\ms{-8mu}}

\DeclareMathOperator{\Br}{Br}
\DeclareMathOperator{\SL}{SL}

\newcommand{\fW}{\mathfrak{W}}

\DeclareMathOperator{\Conf}{Conf}

\newcommand{\rind}[1]{\overrightarrow{\mathfrak{w}}(#1)}

\newcommand{\uh}{\underline{\mathsf{h}}}

\newcommand{\bcol}{\mathrm{color}}

\newcommand{\mD}{\mathsf{D}}
\newcommand{\pP}{\mathcal{P}}
\newcommand{\tpP}{\widetilde{\pP}}

\newcommand{\fl}{\mathfrak{Fl}}
\newcommand{\rxw}{{\underline{\Delta}}}
\newcommand{\PM}{\triangle}
\newcommand{\cmC}{\mathsf{C}}
\newcommand{\KR}{Kirillov–Reshetikhin\ }
\newcommand{\cA}{\scrA}
\newcommand{\ccA}{{\cA_\bullet}}
\newcommand{\seed}{\Sigma}
\newcommand{\gf}{{\g_{\rm fin}}}

\title[\(i\)-boxes and Demazure weaves]{A Comparison of cluster algebra structures arising from \(i\)-boxes and Demazure weaves}

\author[J. Huh]{JiSun Huh}
\thanks{The research of J.\ Huh was supported by the National Research Foundation of Korea(NRF) grant funded by the Korea government (MSIT) (RS-2023-00273425).}
\address[J. Huh]{Department of Mathematics, University of Seoul, Seoul 02504, Republic of Korea}
\email{hyunyjia@yonsei.ac.kr}

\author[W.-S. Jung]{Woo-Seok Jung}
\thanks{The research of W.-S. Jung was supported by the National Research Foundation of Korea (NRF) grants funded by the Korea government (Ministry of Education and MSIT) (RS-2024-00451844 and RS-2023-00273425).}
\address[W.-S. Jung]{Department of Mathematics, University of Seoul, Seoul 02504, Republic of Korea}
\email{jungws@uos.ac.kr}

\author[M. Kim]{Myungho Kim}
 \thanks{The research of M.\ Kim was supported by the National Research Foundation of Korea (NRF) Grant funded by the Korea Government(MSIT)(NRF-2020R1A5A1016126).} 
\address[M. Kim]{Department of Mathematics, Kyung Hee University, Seoul 02447, Republic of Korea}
\email{mkim@khu.ac.kr}

\author[E. Park]{Euiyong Park}
\thanks{The research of E.\ Park was supported by the National Research Foundation of Korea (NRF) Grant funded by the Korea Government(MSIT)(RS-2023-00273425 and NRF-2020R1A5A1016126).}
\address[E.\ Park]{Department of Mathematics, University of Seoul, Seoul 02504, Republic of Korea; 
\&
Department of Mathematics, University of Connecticut, Storrs, CT 06269, USA
}
\email{epark@uos.ac.kr}

\keywords{}
\date{June 22, 2026}

\begin{document}

\maketitle

\begin{abstract}
We compare two cluster algebras related to a positive element \(\ttb\) in the braid group of finite \(ADE\) type. One is the localized bosonic extension \(\tobbA_\C(\ttb)\) equipped with an initial seed arising from an admissible chain \(\frakC\) of \(i\)-boxes, which is deeply connected to monoidal categorification. The other is the coordinate ring \(\C[X(\rxw\uii)]\) of the braid variety \(X(\rxw\uii)\) equipped with an initial seed arising from a Demazure weave $\fW$, where $\uii$ and $\rxw$  are expression sequences of $\ttb$ and  the half twist $\Delta$,  respectively. 
We explicitly construct a Demazure weave \(\fW_{\rxw}(\frakC)\) for each admissible chain \(\frakC\) associated with $\uii$, and prove that there exists an algebra isomorphism \(\varphi_{\uii}\colon \tobbA_\C(\ttb)\to\C[X(\rxw\uii)]\) which is compatible with the two seeds arising from $\frakC$ and \(\fW_{\rxw}(\frakC)\). Moreover, the isomorphism $\varphi_{\uii}$ sends the PBW vectors \(\bFF_{\uii,k} \in \tobbA_\C(\ttb)\) to the coordinates \(z_k \in \C[X(\rxw\uii)]\) indexed by the letters of \(\uii\).
As applications, we investigate a connection between Demazure weaves and signed words via the $i$-boxes and interpret the isomorphism $\varphi_{\uii}$ from the viewpoint of monoidal categorification using Hernandez--Leclerc categories.  
\end{abstract}

\setcounter{tocdepth}{1}
\tableofcontents

\section{Introduction}

Cluster algebras were introduced by Fomin and Zelevinsky in studying total positivity and dual canonical bases of quantum groups (\cite{FZ02}). 
When an algebra has a cluster algebra structure, this structure provides a way to organize distinguished generators into groups called \emph{seeds}, together with procedures called \emph{mutations} to connect seeds.
Since their introduction, cluster algebras have played an important role in several areas of mathematics including representation theory, combinatorics, geometry, and categorification.

Let \(\cmC\) be a Cartan matrix of finite \(ADE\) type with index set \(I\), and let \(\Br^+\) be the positive monoid of the (generalized) braid group $\Br$ associated with \(\cmC\). 
There exist two interesting cluster algebras arising from  $\Br^+,$ 
namely the \emph{bosonic extensions} and the coordinate rings of 
  \emph{braid varieties}.

The bosonic extension \(\hcalA\) is the associative algebra over \(\Q(q^{1/2})\) generated by \(f_{i,p}\) (\(i\in I\) and \(p\in\mathbb Z\)) with  quantum Serre relations and \(q\)-boson relations determined by $\cmC$. This algebra \(\hcalA\) is decomposed into an infinite tensor product of the negative half \(\calU_q^-(\g)\) of the quantum group \(\calU_q(\g)\)  associated with \(\cmC\) as a \(\Q(q^{1/2})\)-vector space, i.e., $\hcalA \simeq \otimes_{\Z} \calU_q^-(\g) $. In this sense, \(\hcalA\) can be understood as an \emph{affinization} of $\calU_q^-(\g)$. Bosonic extensions are motivated by the ring presentations of quantum Grothendieck rings of \emph{Hernandez--Leclerc categories} over quantum affine algebras. 
The Hernandez--Leclerc category is a distinguished monoidal subcategory of the category of finite-dimensional integrable modules over a quantum affine algebra.  
Hernandez and Leclerc investigated such monoidal subcategories from the perspective of cluster algebras, and  studied the quantum Grothendieck rings of the categories and their ring presentations, which motivate the definition of bosonic extensions (\cite{HL10,HL15, HL16}).

Braid symmetries on \(\hcalA\), studied in
\cite{KKOP21B, JLO23, KKOP24E}, define the subalgebra \(\hcalA(\ttb)\) of $\hcalA$ 
generated by the \emph{PBW vectors} \(\FF_{\uii,k}\) associated with an expression sequence \(\uii\)
of a positive braid \(\ttb \in \Br^+\) (\cite{OP25,KKOP24E}).  
 The algebras \(\hcalA(\ttb)\) generalize the \emph{quantum unipotent coordinate rings} \(A_q(\mathfrak n(w))\), a quantum analogue of a coordinate ring of \emph{unipotent subgroups}, where positive braid elements $\ttb$ play the role of Weyl group elements $w$.
The algebra \(\hcalA(\ttb)\) has a (quantum) cluster algebra structure whose initial seeds are described in terms of \(i\)-boxes (\cite{KKOP25}). 
For the purposes of this paper, we consider the specialization  \(\obbA(\ttb)\) of \(\hcalA(\ttb)\) at \(q^{1/2}=1\).  
We choose an expression sequence $\uii=(i_1,i_2, \ldots,i_r)$  of $\ttb$. 
An \emph{\(i\)-box} associated with $\uii$ is an interval \([a,b]\subset[1,r]\) with \(i_a=i_b\), which is a combinatorial object used to describe a seed of \(\obbA(\ttb)\). For each admissible chain \(\frakC=(\frakc_t=[a_t,b_t])_{1\le t\le r}\) of \(i\)-boxes, one can define distinguished elements \(\mD_{\uii}[a,b]\in\obbA(\ttb)\) using the \emph{$T$-systems} and PBW vectors, which gives a seed
\[
    \seed(\frakC)=(\mD(\frakC),Q(\frakC))
\]
for \(\obbA(\ttb)\), where \(\mD(\frakC)=\{\mD_{\uii}[a_t,b_t]\mid 1\le t\le r\}\) and \(Q(\frakC)\) is the quiver determined by \(\frakC\) (see \cite[Theorem~9.4]{KKOP25} and see also Section \ref{Sec: cs for be}).

Braid varieties are algebraic varieties defined from expression sequences of $\ttb$ for the simple algebraic group \(\sfG\) associated with the Cartan matrix $\cmC$ (\cite{CGGLSS25, GLSBS22,GLSB26}). They generalize important families such as open Richardson varieties and double Bott-Samelson varieties. Let us fix an expression sequence \(\uii = (i_1,i_2,\ldots,i_r)\) of $\ttb$, and let \(\delta(\ttb)\) be the element in the Weyl group $\weyl$ obtained from \(\ttb\) by taking the \emph{Demazure product} (see \eqref{Eq: Dp}).
We take a Borel subgroup \(\sfB\) of $\sfG$.
The braid variety \(X(\uii)\) consists of flag configurations \((\sfB_0,\sfB_1, \ldots,\sfB_r)\) in the flag variety \(\fl = \sfG/\sfB\) such that \(\sfB_0=\sfB\), \(\sfB_r=\delta(\ttb)\sfB\)  and the pair \((\sfB_{t-1},\sfB_t)\) is in relative position \(s_{i_t}\) for \(1\le t\le r\), where $s_i$ is the simple reflection of the Weyl group $\weyl$.
 Although a different choice of expression $\uii$ yields an isomorphic variety, the choice of expression is important: it determines ambient affine coordinates for the braid variety, which in turn give a set of defining equations for $X(\uii)$. 

Let \(w_\circ\) be the longest element of $\weyl$, and let \(\Delta\in\Br^+\) denote the \emph{half twist}, i.e., the minimal lift of \(w_\circ\) under the projection $\Br^+ \twoheadrightarrow \weyl$.
Then the braid variety \(X(\rxw\uii)\) is isomorphic to the \emph{half decorated double Bott-Samelson variety} for an expression sequence \(\rxw\) of $\Delta$, where $\rxw \uii$ is the concatenation of $\rxw$ and $\uii$ (see \cite[Lemma~3.16]{CGGLSS25} and see also Section \ref{Sec: cs for bv}).
In this paper, we mainly focus on these braid varieties $ X(\rxw \uii) $.
 Note that the cluster algebra structure on the coordinate ring $\C[X(\rxw(\uii))]$  of the double Bott-Samelson variety is extensively studied in \cite{SW24}, and it provides a starting point for constructing cluster structures on more general braid varieties in \cite{CGGLSS25}.

The coordinate ring \(\C[X(\uii)]\) of a braid variety carries a cluster algebra structure from \emph{Demazure weaves} (\cite{CGGLSS25}). A Demazure weave for \(\uii\) is a planar graph, written \(\fW:\uii\to\delta(\uii)\), whose edges are labeled by the index set $I$ and whose horizontal slices record \emph{weave moves} from \(\uii\) to the Demazure product \(\delta(\uii)=\delta(\ttb)\). 
The weave \(\fW\) provides a seed of \(\C[X(\uii)]\) 
\[
    \Sigma(\fW)=\bigl(\calA(\fW),Q(\fW)\bigr),
\]
where 
$\calA(\fW)$ is the set of the regular functions in \(\C[X(\uii)]\) defined by \(\fW\), and
$Q(\fW)$ is the quiver whose vertices are indexed by trivalent vertices of \(\fW\) and whose arrows are determined by \emph{Lusztig's cycles} (\cite[Sections 4 and 5]{CGGLSS25}).
We remark that cluster structures on braid varieties were also constructed using \(3\)-dimensional plabic graphs and Deodhar geometry (\cite{GLSBS22,GLSB26}). The construction using Deodhar geometry was compared with the construction from Demazure weaves in \cite{CGGLSBS25}.

In this paper, we compare two cluster algebras $\obbA(\ttb)$ and \(\C[X(\rxw \uii)]\) by investigating their seeds arising from $i$-boxes and the associated Demazure weaves, where $\uii$ is an expression sequence of $\ttb$. 
Define \(\obbA_\C(\ttb) :=\C\otimes_\Z\obbA(\ttb)\), and let \(\tobbA_\C(\ttb)\) be the localization of \(\obbA_\C(\ttb)\) by inverting the frozen variables of \(\obbA(\ttb)\) (see 
Section \ref{Sec: cs for be}). 
For each admissible chain \(\frakC\) of \(i\)-boxes associated with $\uii$, we construct the corresponding Demazure weave \(\fW_{\rxw}(\frakC)\) using the combinatorics of $i$-boxes. 
Theorem~\ref{thm: main sec3} says that 
there exists an algebra isomorphism
\[
\varphi_{\uii}: \tobbA_\C(\ttb)\xrightarrow{\sim}\C[X(\rxw\uii)],
\]
which sends the seed $\seed(\frakC)$ to the seed $\Sigma(\fW_{\rxw}(\frakC))$ in the sense that 
$$ 
\calA(\fW_{\rxw}(\frakC)) = \varphi_{\uii}(\mD(\frakC))   \qtq  Q(\fW_{\rxw}(\frakC)) = Q(\frakC)^{\mathrm{op}},
$$
where $Q^{\mathrm{op}}$ is the opposite quiver of $Q$.
Moreover, this morphism $\varphi_{\uii}$ sends the PBW vectors \(\bFF_{\uii,k}\) to the coordinate functions \(z_k\) of $ \C[X(\rxw\uii)]$ corresponding to the letters of \(\uii\)
(see \eqref{Eq: 3 iso} for the coordinates $z_k$). 
We remark that the existence of 
an algebra isomorphism between \(\tobbA_\C(\ttb)\) and \(\C[X(\rxw\uii)]\) compatible with cluster algebra structures
was previously shown by Qin in the context of the cluster algebras associated with signed words (see \cite[Section 7.2]{Qin24A} and \cite[Section 6.3]{Qin24B}).
The main contribution of this paper is to show that the isomorphism \(\varphi_{\uii}\) is compatible with our weave construction $\fW_{\rxw}(\frakC)$ arising from admissible chains $\frakC$ of $i$-boxes, which provides a natural explanation why the PBW vectors of $\tobbA_\C(\ttb)$ correspond to the coordinates of $\C[X(\rxw\uii)]$.

The two cluster algebra structures compared here arise from different constructions: one using bosonic extensions and $i$-boxes, which is deeply connected to monoidal categorification, and the other using braid varieties and Demazure weaves arising in geometry.
The morphism $\varphi_{\uii}$ in Theorem~\ref{thm: main sec3} provides a bridge between them and allows us to compute explicitly the cluster variables of $\fW_{\rxw}(\frakC)$ and the quiver $Q(\fW_{\rxw}(\frakC))$ using the combinatorics of $i$-boxes. Although the weaves $\fW_{\rxw}(\frakC)$ form a special family of weaves, 
both \(\calA(\fW_{\rxw}(\frakC))\) and \(Q(\fW_{\rxw}(\frakC))\) can be computed easily using \(i\)-boxes, as illustrated through examples in the paper.

The proof of Theorem~\ref{thm: main sec3} begins by constructing the weave \(\fW_{\rxw}(\frakC)\) from an admissible chain \(\frakC\).  
 First, we associate a double string \(\bfs_{\rxw}(\frakC)\) with \(\frakC\),  and obtain the corresponding  double inductive weave \(\fW^B_{\rxw}(\frakC)\)  following \cite[Section 6.4]{CGGLSS25}.  
The weave \(\fW_{\rxw}(\frakC)\) is then obtained by vertically concatenating the  weave \(\fW^B_{\rxw}(\frakC)\) with another weave \(\fW^{\rm T}_{\rxw}(\frakC)\) consisting of braid moves chosen so as to make the top equal to  \( \rxw \uii\) (see Section \ref{Sec: diw from iBox} for precise construction).
The argument first treats the distinguished admissible chain \(\frakC_{\calR}\) in \eqref{Eq: C calR}; for this chain, the weave \(\fW^B_{\rxw}(\frakC_{\calR})\) coincides with the right inductive weave \(\rind{\rxw\uii}\). We thus identify \(Q(\rind{\rxw\uii})\) with \(Q(\frakC_{\calR})^{\mathrm{op}}\) (see Lemma~\ref{lem: RRR and riw}), which yields  
the isomorphism \(\varphi_{\uii}\) defined by sending the seed \(\seed(\frakC_{\calR})\) of $\tobbA_\C(\ttb)$ to the seed \(\Sigma(\fW_{\rxw}(\frakC_{\calR}))\) of $\C[X(\rxw\uii)]$.  
Every admissible chain can be reached from \(\frakC_{\calR}\) by \emph{box moves} (see Section \ref{Sec: box moves}). The box moves can be described in terms of mutations and permutations (Proposition~\ref{prop:boxmove_seed}), and Theorem~\ref{thm: modified weave boxmove} gives the same mutations or permutations for the seeds constructed from the weaves \(\fW_{\rxw}(\frakC)\).
A key ingredient is the description in \cite[Proof of Theorem 6.8]{CGGLSS25} of mutations and permutations of inductive weaves arising from the corresponding moves of double strings. 
This explains why the isomorphism \(\varphi_{\uii}\) induces the correspondence between \(\frakC\) and  \(\fW_{\rxw}(\frakC)\) at the level of seeds, and between the PBW vectors \(\bFF_{\uii,k}\) and the coordinates  \(z_k\) at the level of rings.

We further investigate the isomorphism $\varphi_{\uii}$ from the perspective of the monoidal categorification through the representation-theoretic approach using \emph{affine determinantial modules} in \cite{KKOP25}. 
We remark that the monoidal categorification was also studied and shown by Qin via a different method based on a cluster-algebraic approach (see \cite[Theorem 8.19]{Qin24A} and \cite[Theorem 1.6]{Qin24B}). 
Let $\Cgo$ be the \emph{Hernandez--Leclerc} category over a quantum affine algebra and let $\Cg^\bbD(\ttb)$ be the distinguished subcategory of $\Cgo$ categorifying $\obbA(\ttb)$, i.e., 
$$
K( \Cg^\bbD(\ttb)) \simeq  \obbA(\ttb)
$$ 
(see Section \ref{sec: connection to qaa} for precise definitions).  
Under this categorification, for each expression sequence $\uii$ of $\ttb$, 
the PBW vectors \(\bFF_{\uii,t}\) and \(\mD_{\uii}[a,b]\)
correspond to the \emph{affine cuspidal modules} $C_t^\uii$ and the \emph{affine determinantial modules} $M^{ \uii}[a,b]$, respectively. 
Proposition \ref{cor: K simeq oA} tells us that the modules $C_t^\uii$ correspond to the coordinates $z_t$ and the modules $M^{ \uii}[a,b]$ correspond to cluster variables arising from  \(\fW_{\rxw}(\frakC)\) via the isomorphism $\varphi_{\uii}$. In this sense, we categorify the coordinates $z_t$ of $X(\rxw\uii)$ as the
affine cuspidal modules $C_t^\uii$ in the category $\Cg^\bbD(\ttb)$.

As another direction of investigation, we study connections with signed words.
Using the connection between signed words and admissible chains of \(i\)-boxes 
studied in \cite{CQW26}, to each admissible chain \(\frakC\), one assigns a signed word \(\uh(\frakC)\) with \(Q(\uh(\frakC))=Q(\frakC)\).  
Combining this with Theorem~\ref{thm: main sec3}, we obtain
\[
Q(\uh(\frakC))=Q(\fW_{\rxw}(\frakC))^{\rm op}
\]
(see Corollary~\ref{cor:signed-word-weave-quiver}).

The paper is organized as follows. 
Section~\ref{sec:prilim} recalls necessary background on quantum groups and cluster algebras.  Section~\ref{Sec: i-boxes} and Section~\ref{Sec: BE} review the combinatorics of \(i\)-boxes,  bosonic extensions and their cluster algebra structure.  
Section~\ref{Sec: weave} and Section~\ref{sec: braid varieties} discuss Demazure weaves,
braid varieties and their cluster algebra structures.  
Section~\ref{sec: comparison of two cluster algebra structures} constructs the weave \(\fW_{\rxw}(\frakC)\) for an admissible chain $\frakC$ and proves the existence of the isomorphism $\varphi_{\uii}$ with the desired properties. 
Section~\ref{sec: connection to signed words} compares the quivers obtained from \(i\)-boxes, signed words, and Demazure weaves.  
Section~\ref{sec: connection to qaa} investigates the isomorphism $\varphi_{\uii}$ from the perspective of monoidal categorification using Hernandez--Leclerc categories.


\vskip 1em

{\bf Acknowledgments}.\ 
The authors would like to thank Fan Qin for pointing out relevant references concerning the isomorphism \(\varphi_{\uii}\) and for fruitful discussions.
E.P.\ would like to thank the Department of Mathematics at the University of Connecticut for its hospitality and excellent research environment during his visit.

\vskip 1em 

\textbf{Convention}. Throughout this paper, we use the following convention.
\bna
\item For $a,b \in \mathbb{Z} \sqcup \{\pm\infty\}$ with $a \le b$, we define the \emph{interval} $[a,b]$ by
\[
[a,b] := \{ t \in \mathbb{Z} \mid a \le t \le b \}.
\]
If $a>b$, we set $[a,b]=\emptyset$. We sometimes write $[a]$ for $[a,a]$ if no confusion arises.
\item For a set $S$, $|S|$ denotes the cardinality of $S$.
\item Let $\uii=(i_1, i_2, \ldots, i_r)$ and $\ujj=( j_1, j_2, \ldots,j_s)$ 
be sequences with entries in an index set $I$. 
The \emph{concatenation} $\uii * \ujj$ of $\uii$ and $\ujj$ is defined as
\[
\uii * \ujj := (i_1,\ldots,i_r,j_1,\ldots,j_s).
\]
For simplicity, we write $\uii \ujj$ for $\uii * \ujj$ if no confusion arises.
\ee

\vskip 2em 

\section{Preliminaries}\label{sec:prilim}

\subsection{Quantum groups} \label{Sec: Quantum group}
In this subsection, we recall the notion of quantum groups (\cite{LusztigBook}).

Let $(\cmC, \wl, \Pi, \wl^\vee, \Pi^\vee)$ be a Cartan datum consisting of 
a symmetrizable generalized Cartan matrix $\cmC = (c_{i,j})_{i,j\in I}$, 
a weight lattice $\wl$, a coweight lattice $\wl^\vee$, the set of simple roots $\Pi=\{\al_i\}_{i \in I}$, and the set of simple coroots $\Pi^\vee=\{h_i \}_{i \in I}$. They satisfy the pairing condition $\langle h_i, \al_j \rangle = c_{i,j}$ for $i,j\in I$, and there exists a $\mathbb{Q}$-valued symmetric bilinear form $(\cdot,\cdot)$ on $\wl$ such that
\[
\langle h_i, \la \rangle = \frac{2(\al_i, \la)}{(\al_i, \al_i)}
\]
for all $i \in I$ and $\la \in \wl$. 
In this paper, we choose a bilinear form $(\cdot,\cdot)$ such that  $ (\al_i, \al_i)=2$ whenever $\cmC$ is symmetric. 

Let $ q$ be an indeterminate, and let  $\calU_q(\sfg)$ be the quantum group  associated with $\cmC$, which is the associative algebra over $\Q(q)$ generated by $e_i, f_i$ ($i\in I$) and $q^h$ ($h\in \wl^\vee$) satisfying certain condition (see \cite{LusztigBook} for details). We denote by $\calU_q^-(\sfg) = \langle f_i \mid i\in I \rangle$ the negative half of $\calU_q(\sfg)$.

Let $\weyl$ be the Weyl group associated with the Cartan matrix $\sfC$, which is generated by  $s_i$ $(i \in I)$ given as
$$ s_i\la = \la -\lan h_i,\la \ran\al_i \quad \text{for $\la \in \wl$}.  
$$
We denote by $\Br$ the \emph{$($generalized$)$ braid group} associated with $\cmC$. The braid group $\Br$ has a group presentation: the generating set $\{ \sigma_i^{\pm1} \mid i \in I\}$ and its defining relations
\begin{align*} 
\sigma_i \sigma_i^{-1} = \sigma_i^{-1}\sigma_i = {\rm id} 
\qtq
\underbrace{\sigma_i\sigma_j\sigma_i\sigma_j \cdots}_{m(i,j) \text{ times}} = \underbrace{\sigma_j\sigma_i\sigma_j\sigma_i \cdots}_{m(i,j) \text{ times}} \quad \text{ for $i \ne j \in I$},
\end{align*}
where 
\begin{align*} 
m(i,j) \seteq
\begin{cases}
c_{i,j} c_{j, i}+2 & \text{ if } c_{i,j} c_{j, i} \le 2,\\
6 & \text{ if } c_{i,j} c_{j, i} = 3,  \\
\infty & \text{ otherwise.}
\end{cases}
\end{align*}
We set $\Br^+$ to be the submonoid of $\Br$ generated by $\sigma_i$ for all $i\in I$. 
There is a natural group epimorphism $\Br \twoheadrightarrow \weyl$ 
sending $ \sigma_{i}$ to $s_i$ for $i\in I$.
For an element $\ttb \in \Br^+ $ (resp.\ $w \in \weyl$), we denote by $\ell(\ttb)$ (resp.\ $\ell(w)$) the length of $\ttb$ (resp.\ $w$). 
We denote the canonical lift of $w \in \weyl$ to the positive braid monoid $\Br^+$ by $\ttb(w)$. 
When $\weyl$ is of finite type, we define 
\begin{align} \label{Eq: D w0}
\Delta := \ttb(w_\circ) \in \Br^+,
\end{align}
where $w_\circ$ is the longest element in $\weyl$. We call $\Delta$ the \emph{half twist}. For $i\in I$, we define $i^*$ to be the index defined as follows:
\begin{align} \label{Eq:istar}
\al_{i^*} = - w_\circ (\al_i).
\end{align}
Note that the correspondence $i \mapsto i^*$ gives an involution on the index set $I$.

We now define expression sequences of $\ttb \in \Br^+$ and $w \in \weyl$ as follows.
\bnum
\item A sequence $\uii=(i_1,i_2,\ldots,i_r) \in I^{r}$ is said to be an \emph{expression} (or \emph{sequence}) of $\ttb$ if $\ttb = \sigma_{i_1}\sigma_{i_2}\cdots \sigma_{i_r}$. 
\item A sequence $\uii=(i_1,\ldots,i_{r}) \in I^{r}$ is said to be a \emph{reduced expression} (or \emph{reduced sequence}) of $w$
if $w = s_{i_1}\cdots s_{i_{r}}$ and $ \ell(w) = r $. 
\ee
We denote by $\Seq(\ttb)$ (resp.\ $\Seq(w)$) the set of all expressions of $\ttb$ (resp.\ $w$). We often write $\rxw$ for an expression of $\Delta$ in $\Seq(\Delta)$.

\subsection{Cluster algebras}\label{subsec: cluster algebra}
In this subsection, we recall the definition of a cluster algebra (\cite{FZ02}). 

Let $K$ be a finite index set with the decomposition $K = K^{\mathrm{ex}} \sqcup K^{\mathrm{fr}}$. We call $K^{\mathrm{ex}}$ the set of \emph{exchangeable indices} and $K^{\mathrm{fr}}$  the set of \emph{frozen indices}. Let $\frakF$ be a field of characteristic 0.

A \textit{seed} is a pair $\seed = (\mathbf{x}, B)$ consisting of 
\begin{itemize}
\item a set $\mathbf{x} = \{ x_k \mid k \in K \}$   
    algebraically independent in $\frakF$,
\item an integer matrix 
    \[
    B = (b_{i,j})_{(i,j) \in K \times K^{\mathrm{ex}}}
    \]
    whose principal part, namely the submatrix indexed by
    $K^{\mathrm{ex}} \times K^{\mathrm{ex}}$, is \emph{skew-symmetric}.
\end{itemize}

The exchange matrix $B$ determines a quiver $Q$ with no loops or 2-cycles.
The vertex set of $Q$ is $K$, and the arrows of $Q$ are determined by $B$ as follows:
\begin{itemize}
\item if $b_{i,j} > 0$, then there are $b_{i,j}$ arrows from $i$ to $j$,
\item if $b_{i,j} < 0$, then there are $-b_{i,j}$ arrows from $j$ to $i$,
\item if $b_{i,j} = 0$, then there are no arrows between $i$ and $j$.
\end{itemize}
In this setting, vertices in $K^{\mathrm{fr}}$ are called \emph{frozen vertices},
and vertices in $K^{\mathrm{ex}}$ are called \emph{mutable vertices}. We depict
$K^{\mathrm{fr}}$ by rectangular nodes and $K^{\mathrm{ex}}$ by circular nodes. 
We write $B(\Sigma)$ and $Q(\Sigma)$ for the above $B$ and $Q$ respectively, and 
we sometimes write such a seed by $\Sigma = (\mathbf{x}, Q)$ instead of $\Sigma = (\mathbf{x}, B)$. 

A \emph{labeled quiver} $Q$ is a quiver together with an assignment $l : Q_0 \to \calL$,
where $Q_0$ is the set of vertices of $Q$ and $\calL$ is the set of labels. Unless otherwise specified for $\calL$, we understand $\calL = \{ 1,2, \ldots, m \}$, where $m = |Q_0|$.
For a (labeled) quiver $Q$,  we denote by $Q^{\rm op}$ the \emph{opposite quiver} of $Q$, 
which is the quiver defined by keeping the same set of vertices (and the same labeling) of $Q$ and reversing all arrows of $Q$. Note that, if $B$ is the exchange matrix corresponding to $Q$, then $-B$ is the exchange matrix corresponding to $Q^{\rm op}$.

For any exchangeable index $k \in K^{\mathrm{ex}}$, the \textit{mutation} in direction $k$, denoted by $\mu_k$, transforms a seed $\seed = (\mathbf{x}, B)$ into a new seed $\Sigma' = \mu_k(\Sigma) := (\mathbf{x}', B')$. 
The entries of the mutated exchange matrix $B' = (b'_{i,j})$ are given by
\begin{equation*}
    b'_{i,j} =
    \begin{cases}
        - b_{i,j} & \text{if } i = k \text{ or } j = k, \\[2pt]
        b_{i,j} + \dfrac{|b_{i,k}| b_{k,j} + b_{i,k} |b_{k,j}|}{2}
        & \text{otherwise},
    \end{cases}
\end{equation*}
and the cluster variables $\mathbf{x}' = \mu_k(\mathbf{x}) = \{ x'_j \mid j \in K \}$
are updated according to the \emph{exchange relation}:
\begin{equation*}
    x'_j = 
    \begin{cases}
        \displaystyle{x_k^{-1} \left( \prod_{\substack{i \in K \\ b_{i,k}>0}} x_i^{b_{i,k}} 
        + \prod_{\substack{i \in K \\ b_{i,k}<0}} x_i^{-b_{i,k}} \right)} & \text{if } j=k, \\
        x_j & \text{if } j \neq k.
    \end{cases}
\end{equation*}
Two seeds are mutation equivalent if they are connected by a finite sequence of mutations.

We fix an initial seed $\seed = (\mathbf{x}, B)$ (equivalently, $\Sigma = (\mathbf{x}, Q)$).
Let $\mathbb{P} = \C  [x_j^{\pm 1} \mid j \in K^{\mathrm{fr}}]$ be the ring of Laurent polynomials in the frozen variables, and let $\overline{\mathbb{P}} =  \C [x_j \mid j \in K^{\mathrm{fr}}]$ be the ring of polynomials in the frozen variables. 

\begin{definition} \label{Def: cluster alg} \

\bnum
\item  The \emph{cluster algebra} $\cA(\seed)$ is the $\mathbb{P}$-subalgebra of the ambient field $\frakF$ generated by all cluster variables in all seeds which are mutation equivalent to $\seed$.
\item \label{Def: cluster alg (ii)}
We denote by $\ccA(\seed)$  the $\overline{\mathbb{P}}$-subalgebra of $\frakF$ generated by all cluster variables in all seeds which are mutation equivalent to $\seed$. This algebra is called the \emph{partially compactified cluster algebra} (\cite{GHKK18}). We understand $ \ccA(\seed) $ as a subalgebra of $ \cA(\Sigma)$. Note that $\cA(\seed)$ can be obtained from $\ccA(\seed)$ by localizing at the set of frozen variables.
\ee
\end{definition}

For an algebra $A$ in $\frakF$, we say that $A$ has a \emph{cluster algebra structure} if there exists a family $F$ of seeds $\seed = (\mathbf{x}, B)$ of $A$, i.e., $\mathbf{x} \subset A$, such that 
\bna
\item for any seed $\seed \in F$, the cluster algebra $\cA(\seed)$ (resp.\ $\ccA(\seed)$) coincides with $A$,
\item any mutation of a seed in $F$ is in $F$, 
\item any pair of seeds $\seed$ and $\seed'$ in $F$ are connected by a finite sequence of mutations.
\ee

\vskip 2em 

\section{Combinatorics of $i$-boxes} \label{Sec: i-boxes}

In this section, we review the notion of $i$-boxes following \cite{KK24, KKOP25,KKOP24A}.
An $i$-box is a combinatorial object that yields both of the seeds for the bosonic extensions  and the monoidal seeds for the Hernandez--Leclerc categories. 

Let $\cmC$ be a Cartan matrix of \emph{finite $ADE$ type} with an index set $I$, and let $\Br$ be the braid group associated with $\cmC$.
Throughout this section, we choose and fix an expression sequence 
\begin{align}
\uii = (i_1, i_2, \ldots, i_r) \in \Seq(\ttb)
\end{align}
for $\ttb \in \Br^+$.

\subsection{Chains of $i$-boxes} \label{Sec: chains of iboxes}
For an index $1 \le k \le r$ and $j \in I$, we define the nearest indices
at which $j$ occurs to the right and to the left of $k$ by
\begin{align*}
k_\uii(j)^+ &:= \min \bl \{ t \in [1,r] \mid t\ge k,\; i_t=j \} \sqcup \{ +\infty \} \br, \\
k_\uii(j)^- &:= \max \bl \{t \in [1,r] \mid t\le k,\; i_t=j \} \sqcup \{-\infty\} \br.
\end{align*}
In addition, we define the strictly next and strictly previous indices
at which $i_k$ occurs by
\begin{align*}
k_\uii^+ &:= \min \bl \{ t \in [1,r]\mid t> k,\; i_t= i_k \} \sqcup \{+\infty\} \br, \\
k_\uii^- &:= \max \bl \{t \in [1,r] \mid t< k,\; i_t=i_k \} \sqcup \{-\infty\} \br. 
\end{align*}
We sometimes omit the subscript $\uii$ when the expression is clear from the context.

\begin{definition}\label{def:ibox}\ 
 \bnum
  \item 
  An interval $\frakc=[a,b] \subseteq [1,r]$ is called an \emph{$i$-box of $\uii$} if $a\le b$ and $i_a=i_b$,
in which case we define its \emph{color} by $\bcol(\frakc):=i_a$.

      \item For a finite interval $[a,b] \subseteq [1,r]$, we define two associated $i$-boxes of $\uii$ by
      \[
      [a,b \} := [a,b(i_a)^-] \quad \text{and} \quad \{a,b] := [a(i_b)^+,b].
      \]
      \item A chain $\frakC = (\frakc_t=[a_t,b_t])_{1 \le t \le r}$ of $i$-boxes of $\uii$
	is called \emph{admissible} if, for each $t=1,\ldots,r$, the union
	\[
	\tfrakc_t = [\ta_t,\tb_t] := \bigcup_{1 \le j \le t} [a_j,b_j]
	\]
	is a finite interval with cardinality $t$, and the $t$-th $i$-box satisfies
	\[
	[a_t,b_t] = [\ta_t,\tb_t\}
	\quad\text{or}\quad
	[a_t,b_t] = \{\ta_t,\tb_t].
	\]
	We call $\tfrakc_t$ the envelope of $\frakc_t$, and $\tfrakc_r$ the range of $\frakC$.
  \ee
\end{definition}

An admissible chain $\frakC=(\frakc_t)_{1\le t\le r}$ is uniquely determined by
its initial envelope $\tfrakc_1=[c,c]$ and its sequence of \emph{horizontal moves}
$\frakH=(\calH_1,\ldots,\calH_{r-1})\in\{\calL,\calR\}^{r-1}$, 
where $c$ equals the number of occurrences of $\calL$ in $\frakH$ plus one. 
We often shortly write $\frakH$ for the pair $(c, \frakH)$ when no confusion arises. 
For $t\ge2$, the envelope $\tfrakc_t$ is obtained from
$\tfrakc_{t-1}=[\ta_{t-1},\tb_{t-1}]$ by
\[
\tfrakc_t=
\begin{cases}
[\ta_{t-1}-1,\tb_{t-1}] & \mbox{if $\calH_{t-1}=\calL$ (\emph{left move})},\\
[\ta_{t-1},\tb_{t-1}+1] & \mbox{if $\calH_{t-1}=\calR$ (\emph{right move})}.
\end{cases}
\]
We call $\frakH$ the \emph{LR sequence} of $\frakC$.
The unique element $z_t\in\tfrakc_t\setminus\tfrakc_{t-1}$ is called the
\emph{effective end} of $\frakc_t$; equivalently,
$z_t=\ta_t$ if $\calH_{t-1}=\calL$ and $z_t=\tb_t$ if $\calH_{t-1}=\calR$. 
We often denote $z_t$ by $z_{\frakc_t}$ when we wish to emphasize the
dependence on the $i$-box $\frakc_t$ rather than on the position $t$ in the chain.
Note that the $i$-box $\frakc_t$ is given by
\[
\frakc_t =
\begin{cases}
[\ta_t,\tb_t\} & \text{if } \calH_{t-1}=\calL,\\
\{\ta_t,\tb_t] & \text{if } \calH_{t-1}=\calR.
\end{cases}
\]

\subsection{Box moves} \label{Sec: box moves}

Let $\frakC=(\frakc_t)_{1\le t\le r}$ be the admissible chain of $i$-boxes associated with $\uii$ determined by a pair $(c,\frakH)$ of $c\in [1,r]$ and an LR sequence $\frakH$. 

For $1\leq k < r$, we say that $\frakc_k$ is \emph{movable} if $k=1$, or if $k\ge2$ and $\calH_{k-1}\neq\calH_k$. For a movable $\frakc_k$ in $\frakC$, the \emph{box move}
of $\frakC$ at $k$ is the admissible chain $\calB_k(\frakC)$ which is determined by the pair $(c',\frakH')$
such that
\[
{\rm (a)}~~ c'=\begin{cases}
c-1 & \mbox{if $k=1$ and $\calH_1=\calL$},\\
c+1 & \mbox{if $k=1$ and $\calH_1=\calR$},\\
c & \mbox{if $k>1$},
\end{cases}
\quad
{\rm (b)}~~ \calH'_s=\begin{cases}
\calR & \mbox{if $s\in \{ k-1,k \}$ and $\calH_s=\calL$},\\
\calL & \mbox{if $s\in \{ k-1,k \}$ and $\calH_s=\calR$},\\
\calH_s & \mbox{if $s\not\in \{ k-1,k \}$}.
\end{cases}
\]
That is, $\calB_k(\frakC)$ is obtained from $\frakC$ by moving $\tfrakc_k$ by $1$ to the right 
or to the left inside $\tfrakc_{k+1}$.

We now construct an exchange matrix associated with the admissible chain $\frakC=(\frakc_t)_{1\le t\le r}$. 
We define the index sets as follows:
\begin{align} \label{eq: frakC index}
\sfK(\frakC) &:= \{1,2,\ldots,r\}, \nonumber\\
\sfK(\frakC)^{\mathrm{fr}}
&:= \{\, t \in  \sfK (\frakC) \mid
\frakc_t = [1(i)^+,\, r(i)^-] \text{ for some } i \in I \,\},\\
\sfK(\frakC)^{\mathrm{ex}}
&:= \sfK(\frakC) \setminus  \sfK (\frakC)^{\mathrm{fr}}. \nonumber
\end{align}
Following \cite{KK24}, we associate to an admissible chain $\frakC$ an exchange
matrix $B(\frakC)$ and the corresponding quiver $Q(\frakC)$.

\begin{definition}[{\cite[\S 3.2]{KK24}}] \label{def: mon_seed_ibox}
Let $\frakC = (\frakc_t)_{1 \le t \le r}$ be an admissible chain of $i$-boxes
associated with $\uii \in \Seq(\ttb)$.
\bnum
\item The matrix  
$\bar{B}(\frakC) = (b_{\frakc_p,\frakc_q})_{(p,q) \in \sfK(\frakC) \times \sfK(\frakC)}$
is defined as follows:
\bna
\item $\bar{B}(\frakC)$ is skew-symmetric, 
\item the entries of $\bar{B}(\frakC)$ are contained in $\{-1,0,1\}$, 
\item for $[x,y],[x',y']\in\frakC$, 
$b_{[x,y],[x',y']}=1$ if and only if one of the following holds:
\begin{itemize}
\item $x=x'$ and $y'=y^{-}$, or $y=y'$ and $x'=x^{-}$;
\item $(\alpha_{i_x},\alpha_{i_{x'}})=-1$ and one of the following conditions holds:
\begin{enumerate}
\item $[x,y^{+}]\in\frakC$, $z_{[x,y]}=x$,
      $x'^{-}<x<x'$, and $y'<y^{+}<y'^{+}$;
\item $[x,y^{+}]\in\frakC$, $z_{[x',y']}=y'$,
      $x'^{-}<x$, and $y<y'<y^{+}<y'^{+}$;
\item $[x'^{-},y']\in\frakC$, $z_{[x',y']}=y'$,
      $x^{-}<x'^{-}<x$, and $y<y'<y^{+}$;
\item $[x'^{-},y']\in\frakC$, $z_{[x,y]}=x$,
      $x^{-}<x'^{-}<x<x'$, and $y'<y^{+}$.
\end{enumerate}
\end{itemize}
\ee

\item 
We define
\[
B(\frakC) := \bar{B}(\frakC)\big|_{\sfK(\frakC)\times \sfK(\frakC)^{\mathrm{ex}}}.
\]
\item We denote by $Q(\frakC)$ the labeled quiver associated with this exchange matrix $B(\frakC)$ whose vertices are labeled by $\{1,2,\dots,r\}$.
The vertex $t$ corresponds to the $i$-box $\frakc_t$.
\ee
\end{definition}

As a special case, we consider the admissible chain of $i$-boxes associated with $\uii$
\begin{align} \label{Eq: C calR}
\frakC_{\calR}
:= (\frakc_t)_{1\le t\le r}
\end{align}
determined by the LR sequence $(\calR,\calR,\ldots,\calR)$.
Then $\frakc_t = \{ 1, t ]$ for $t\in [1,r]$ and the corresponding sequence of colors coincides with $\uii$, i.e., 
\[
(\bcol(\frakc_t))_{1\le t\le r} = \uii.
\]
In this case, the exchange matrix is given by
\begin{align*}
B(\frakC_{\calR})
=(b_{s,t})_{(s,t)\in \sfK(\frakC_{\calR})\times \sfK(\frakC_{\calR})^{\mathrm{ex}} }
\end{align*}
where the entries $b_{s,t}$ are given as follows:
\begin{align} \label{eq: exchange matrix for RRR}
b_{s,t} := 
\begin{cases}
1 & \text{if either }  s=t^+ \text{ or } \bigl( s < t < s^+ < t^+ \text{ and } (\al_{i_s},\al_{i_t})=-1 \bigr),\\[2pt]
-1 & \text{if either }  t=s^+ \text{ or } \bigl( t < s < t^+ < s^+ \text{ and } (\al_{i_s},\al_{i_t})=-1 \bigr),\\[2pt]
0 & \text{otherwise.}
\end{cases}
\end{align}
Note that the quiver $Q(\frakC_{\calR})$ is the GLS quiver associated with $\uii$ (see \cite[(7.5)]{KKOP24A}).

\begin{example}\label{ex: ibox exchange matrix}
Let $\g$ be the simple Lie algebra of type $A_3$ with index set $I=\{1,2,3\}$, and set 
$$
\uii= (2,3,1,2,2,1).
$$
\bnum
\item \label{ex: ibox exchange matrix (i)}
Let $\frakC_{\calR} = (\frakc_t')_{1\le t\le 6}$ be the admissible chain determined by 
$\frakH' = (\calR,\calR,\calR,\calR,\calR)$. Then the $i$-boxes are given as
\begin{align*}
\frakc'_1=[1,1],\quad
\frakc'_2=[2,2],\quad
\frakc'_3=[3,3],\quad
\frakc'_4=[1,4],\quad
\frakc'_5=[1,5],\quad
\frakc'_6=[3,6], 
\end{align*}
and the index sets are given as 
$$
\sfK( \frakC_{\calR} ) = \{ 1,2,3,4,5,6 \}, \quad
\sfK( \frakC_{\calR} )^{\mathrm{ex}} = \{ 1,3,4 \}, \quad  
\sfK( \frakC_{\calR} )^{\mathrm{fr}} = \{ 2,5,6 \}. 
$$
Then the exchange matrix $B(\frakC_{\calR})$ and its quiver $Q(\frakC_{\calR})$ are given as follows:
$$
B(\frakC_{\calR})
=
\begin{pmatrix}
0& 1& -1\\
-1& 0& 0\\
-1& 0& 0\\
1&0& 0\\
0& -1&1\\
0&1&0
\end{pmatrix},
\qquad 
Q(\frakC_{\calR}) =
\raisebox{-13mm}{
\begin{tikzpicture}[scale=0.6]
    \coordinate (1) at (0,0);
    \coordinate (2) at (0,-4);
    \coordinate (3) at (-2,-2);
    \coordinate (4) at (2,-2);
    \coordinate (5) at (2,-4);
    \coordinate (6) at (-2,-4);

    \node[circle, draw=black, minimum size=0.5cm] at (1) {1};
    \node[rectangle, draw=black, minimum size=0.5cm] at (2) {2};
    \node[circle, draw=black, minimum size=0.5cm] at (3) {3};
    \node[circle, draw=black, minimum size=0.5cm] at (4) {4};
    \node[rectangle, draw=black,  minimum size=0.5cm] at (5) {5};
    \node[rectangle, draw=black, minimum size=0.5cm] at (6) {6};

    \draw[->,  thick, shorten >=10pt, shorten <=10pt] (1.south) -- (2.north);            
    \draw[->,  thick, shorten >=10pt, shorten <=10pt] (1.south west) -- (3.north east);  
    \draw[->,  thick, shorten >=10pt, shorten <=10pt] (4.north west) -- (1.south east);  
    \draw[->,  thick, shorten >=10pt, shorten <=10pt] (5.south west) -- (4.north east);  
    \draw[->,  thick, shorten >=10pt, shorten <=10pt] (3.north east) -- (5.west);        
    \draw[->,  thick, shorten >=10pt, shorten <=10pt] (6.south west) -- (3.north);       
\end{tikzpicture}
}.
$$

\item \label{ex: ibox exchange matrix (ii)}
Let  $\frakC=(\frakc_t)_{1\le t\le 6}$ be the admissible chain determined by $\frakH = (\calL,\calR,\calL,\calR,\calR)$. 
Then the $i$-boxes are given as
\begin{align*}
\frakc_1=[3,3],\quad
\frakc_2=[2,2],\quad
\frakc_3=[4,4],\quad
\frakc_4=[1,4],\quad
\frakc_5=[1,5],\quad
\frakc_6=[3,6], 
\end{align*}
and the index sets are given as 
$$
\sfK(\frakC) = \{ 1,2,3,4,5,6 \}, \quad
\sfK(\frakC)^{\mathrm{ex}} = \{ 1,3,4 \}, \quad  
\sfK(\frakC)^{\mathrm{fr}} = \{ 2,5,6 \}. 
$$
Note that $\frakC$ can be obtained from $\frakC_\calR$ by applying box moves, i.e.,  
\[
\frakC=\calB_1( \calB_3 (\calB_2(\calB_1(\frakC_{\calR})))).
\]
Then the exchange matrix $B(\frakC)$ and its quiver $Q(\frakC )$ are given as follows:
$$
B(\frakC)
=\begin{pmatrix}
0&1& -1\\
0& 1& -1\\
-1& 0& 1\\
1&-1& 0\\
-1& 0&1\\
1&0&0
\end{pmatrix},
\qquad 
Q(\frakC ) =
\raisebox{-13mm}{
\begin{tikzpicture}[scale=0.6]
    \coordinate (3) at (0,0);
    \coordinate (2) at (0,-4);
    \coordinate (1) at (-2,-2);
    \coordinate (4) at (2,-2);
    \coordinate (5) at (2,-4);
    \coordinate (6) at (-2,-4);

    \node[circle, draw=black, minimum size=0.5cm] at (3) {3};
    \node[rectangle, draw=black, minimum size=0.5cm] at (2) {2};
    \node[circle, draw=black, minimum size=0.5cm] at (1) {1};
    \node[circle, draw=black, minimum size=0.5cm] at (4) {4};
    \node[rectangle, draw=black,  minimum size=0.5cm] at (5) {5};
    \node[rectangle, draw=black, minimum size=0.5cm] at (6) {6};

    \draw[->,  thick, shorten >=10pt, shorten <=10pt] (1) -- (3);            
    \draw[->,  thick, shorten >=10pt, shorten <=10pt] (1) -- (5);  
    \draw[->,  thick, shorten >=10pt, shorten <=10pt] (2) -- (3);
    \draw[->,  thick, shorten >=10pt, shorten <=10pt] (3) -- (4);
    \draw[->,  thick, shorten >=10pt, shorten <=10pt] (4) -- (1);
    \draw[->,  thick, shorten >=10pt, shorten <=10pt] (4) -- (2);    
    \draw[->,  thick, shorten >=10pt, shorten <=10pt] (5) -- (4);
    \draw[->,  thick, shorten >=10pt, shorten <=10pt] (6) -- (1);    
\end{tikzpicture}
}.
$$
\ee
\end{example}

\vskip 2em 

\section{Bosonic extensions for quantum groups} \label{Sec: BE}

In this section, we recall the bosonic extension for quantum groups (\cite{HL15, OP25, JLO24, KKOP25A, KKOP24E}) and its cluster algebra structure arising from $i$-boxes.

\subsection{Bosonic extensions}
Let $\cmC = (c_{i,j})_{i,j\in I}$ be a Cartan matrix of \emph{finite $ADE$ type} and let $\calU_q(\g)$ be the corresponding quantum group. We denote by $\Br = \ang{ \bg_i \mid i \in I }$ the braid group of $\g$.

For integers $m \ge n \ge 0$, we define the $q$-integers, factorials, and binomial coefficients by
\[
[n]_q := \frac{q^n-q^{-n}}{q-q^{-1}}, \qquad [n]_q! := \prod_{t=1}^n [t]_q, \qquad \begin{bmatrix} m\\ n\end{bmatrix}_q := \frac{[m]_q!}{[n]_q!\,[m-n]_q!}.
\]
The \emph{bosonic extension} $\hcalA$ of $\calU_q^-(\g)$ is the $\kk$-algebra generated by $\{ f_{i,p} \}_{(i,p) \in I \times \Z}$ subject to the following relations:
for any $i,j \in I$ and $m,p \in \Z$,
\bna
\item \label{it: relation1}
$ \displaystyle
\sum^{1-c_{i,j} }_{t=0} (-1)^t  \begin{bmatrix} 1-c_{i,j} \\ t \end{bmatrix}_{q}  {f_{i,p}}^{1-c_{i,j} -t} f_{j,p} f_{i,p}^{t} = 0 
$ \qquad  if $i \ne j$, 
\item \label{it: relation2} 
$
f_{i,m} f_{j,p} =  q^{(-1)^{p-m+1}(\al_i,\al_j)}  f_{j,p} f_{i,m} + \delta_{i,j}\delta_{p,m+1}(1- q^{2})
$ \qquad  if $m < p$.
\ee
We define the subalgebras $\hcalA_{\le k}$ (resp.\ $\hcalA_{\ge k}$) as the $\mathbb{Q}(q^{1/2})$-subalgebras of $\hcalA$ generated by $\{f_{i, m} \mid i \in I, m \le k\}$ (resp.\ $\{f_{i,m} \mid i \in I, m \ge k\}$), and set $ \hcalA_{< k} := \hcalA_{\le k-1}$ and $ \hcalA_{> k} := \hcalA_{\ge k+1}$.

For each $i \in I$, we denote by $\TT_i\in\Aut_\kk(\hcalA)$ the $\kk$-algebra automorphism defined by
\begin{align*}
\TT_i(f_{j,m}) = 
\begin{cases}
    f_{j, m+\delta_{i,j}} & \text{if } (\al_i,\al_j) \ne -1, \\
    \dfrac{ q^{1/2} f_{j,m}f_{i,m} - q^{-1/2}f_{i,m}f_{j,m}  }{q - q^{-1}}  & \text{if } (\al_i,\al_j) = -1.
\end{cases}
\end{align*} 
As the family $\{\TT_i\}_{i\in I}$ satisfies the braid relations (\cite{JLO23, KKOP21B, KKOP24E}),
one can write 
\[
\TT_\ttb:=\TT_{i_1}^{\ep_1}\cdots \TT_{i_r}^{\ep_r}
\]
for any $\ttb\in\Br$ written as $\ttb=\bg_{i_1}^{\ep_1}\cdots \bg_{i_r}^{\ep_r}$ with $\ep_t\in\{\pm1\}$.

Let  $\ttb\in\Br^+$ and define 
\[
\hcalA(\ttb):=\hcalA_{\ge 0}\cap \TT_\ttb(\hcalA_{<0}).
\]
Let us take $\uii=(i_1,i_2,\dots,i_r)\in\Seq(\ttb)$, and define
\[
\FF_{\uii, t} := q^{1/2} \TT_{i_1}\cdots \TT_{i_{t-1}}(f_{i_t,0}) \quad \text{ for } 1\le t \le r.
\]
It was shown in \cite{KKOP24E,OP25} that $\hcalA(\ttb)$ is generated by $\FF_{\uii, t}$ for $t\in [1,r]$ as an algebra. The generators $\FF_{\uii, t}$ are called \emph{PBW vectors} of $\hcalA(\ttb)$ with respect to $\uii$ (see \cite{OP25, KKOP24E} for details). We define $\hcalA_\bfA(\ttb)$ to be the $\bfA$-vector space generated by the ordered products
$\FF_{\uii, r}^{a_r} \FF_{\uii, r-1}^{a_{r-1}} \cdots \FF_{\uii, 1}^{a_1}$
for all $(a_1, \ldots, a_r) \in \Z_{\ge0}^{r}$. Since $\hcalA_\bfA(\ttb)$ forms a ring (see \cite[Proposition 7.2]{KKOP25A}), we define the specialization of $\hcalA(\ttb)$ at $q^{1/2}=1$ by
$$
\obbA(\ttb) := \hcalA_\bfA(\ttb) \big/ (q^{1/2}-1)\hcalA_\bfA(\ttb),
$$
and set $\obbA_\C(\ttb) := \C \otimes_{\Z} \obbA(\ttb)$.
We set 
\begin{align} \label{Eq: PBW vector}
\bFF_{\uii, t} := \FF_{\uii, t}|_{q^{1/2}=1} \in \obbA(\ttb) \qquad \text{for $t \in [1,r]$}.
\end{align}
It was shown in \cite{KKOP24E,OP25} that $\obbA(\ttb)$ is the polynomial ring with the generators 
$\bFF_{\uii, t}$ for $t \in [1,r]$.

\subsection{Cluster structure of $\obbA(\ttb)$} \label{Sec: cs for be}
Recall the $i$-boxes and related notations in Section \ref{Sec: i-boxes}.

Let $\ttb \in \Br^+$ and take $\uii=(i_1,i_2,\ldots,i_r) \in \Seq(\ttb)$. 
For any $i$-box $[a,b]$ of $\uii$, we define the element $\mD_{\uii}[a,b] \in \obbA(\ttb)$ inductively as follows:
\begin{equation} \label{Def: minors}
\begin{aligned}
& \mD_{\uii}[k] := \bFF_{\uii, k} \qquad \text{ for any $k\in [1, r]$}, \\
& \mD_{\uii}[a,b]  \mD_{ \uii}[a^+,b^-] = \mD_{\uii}[a^+,b]  \mD_{ \uii}[a,b^-]
- \prod_{ \substack{ j \in I \\ (\al_{i_a},\al_j)=-1}}
\mD_{\uii}[a(j)^+,b(j)^-].
\end{aligned}
\end{equation}
We use the convention that $\mD_{\uii}[\emptyset]=1$. 
Note that the second relation in \eqref{Def: minors} comes from the $T$-systems among affine determinantial modules in the Hernandez--Leclerc category (see Theorem~\ref{thm:Tsystem} and Proposition \ref{cor: K simeq oA}), which 
is a generalization of $T$-systems among \KR modules. For simplicity, we write $ \mD[a,b]$ and $\bFF_{ k}$ for $\mD_{\uii}[a,b]$ and $\bFF_{\uii, k}$, respectively, if no confusion arises.

Let $\frakC=(\frakc_t)_{1\le t\le r}$ be an admissible chain of $i$-boxes associated with $\uii$. 
We now consider the index sets $\sfK(\frakC)$, 
$\sfK(\frakC)^{\mathrm{fr}}$ and $\sfK(\frakC)^{\mathrm{ex}}$ given in \eqref{eq: frakC index}, and define 
\begin{align} \label{eq: frakC cluster}
\sfD(\frakC) &:= \{\, \mD(\frakc_t) \mid t \in \sfK(\frakC) \,\}
\subset  \obbA(\ttb). 
\end{align}

\begin{theorem}[{\cite[Theorem 9.4]{KKOP25}}] \label{thm: mCat of cluster}
The algebra $\obbA(\ttb)$ has a cluster algebra structure with an initial seed 
\[
\seed(\frakC) := \bigl( \mD(\frakC), B(\frakC)\bigr) 
\]
for any admissible chain $\frakC$ of $i$-boxes associated with $\uii \in \Seq(\ttb)$. Here $B(\frakC)$ is the exchange matrix arising from $\frakC$ in Definition \ref{def: mon_seed_ibox}.
\end{theorem}

The box moves $\calB_k$ of $\frakC$ can be understood as mutations (see \cite[Section 5.2]{KKOP24A} and
also \cite[Section 2.1]{CQW26}).

\begin{proposition}[{\cite[Section 5.2]{KKOP24A}}] \label{prop:boxmove_seed}
Let $\frakC$ be an admissible chain of $i$-boxes and let $\frakc_k$ be a movable
$i$-box in $\frakC$. 
Then the seeds $\seed(\frakC)$ and $\seed(\calB_k(\frakC))$ are related by 
$$
\seed(\calB_k(\frakC))=
\begin{cases}
\mu_k(\seed(\frakC)) & \mbox{if $\bcol(\frakc_k)=\bcol(\frakc_{k+1})$}, \\
\sigma_{k,k+1}(\seed(\frakC)) & \mbox{if $\bcol(\frakc_k)\neq\bcol(\frakc_{k+1})$},
\end{cases}
$$
where $\mu_k$ denotes the mutation at vertex $k$ and $\sigma_{k,k+1}$ denotes the permutation of the vertices labeled $k$ and $k+1$. 
In particular, we have
\[
   B(\calB_k(\frakC)) =
        \begin{cases}
            \mu_k(B(\frakC)) & \mbox{if $\bcol(\frakc_k)=\bcol(\frakc_{k+1})$}, \\
            \sigma_{k,k+1}(B(\frakC)) & \mbox{if $\bcol(\frakc_k)\neq\bcol(\frakc_{k+1})$}.
        \end{cases}
\]

\end{proposition}

\begin{remark}
In the case for $\bcol(\frakc_k)\neq\bcol(\frakc_{k+1})$ in Proposition \ref{prop:boxmove_seed}, the clusters $\mD(\frakC)$ and $\mD(\calB_k(\frakC))$ differ by $\sigma_{k,k+1}$ as ordered sets, but they coincide as usual sets. 
\end{remark}

Thanks to Theorem \ref{thm: mCat of cluster}, we have an isomorphism of $\C$-algebras
\[
 \ccA(\seed(\frakC)) \simeq \obbA_\C(\ttb),
\]
where $\ccA(\seed(\frakC))$ is the partially compactified cluster algebra over $\C$ in Definition \ref{Def: cluster alg}.

For any $\uii = (i_1, \ldots, i_r) \in \Seq(\ttb)$, we set  
\begin{align*} 
\pP(\uii) := \C[x_{1}, x_{2}, \ldots ,x_{r}]
\end{align*}
to be the polynomial ring in $r$ variables $x_{k}$ ($k\in [1,r]$). Then the PBW vectors of $\obbA(\ttb)$ associated with $\uii$ give an isomorphism 
\begin{align*}
\psi_{\uii} : \pP(\uii) \buildrel \sim \over \longrightarrow \obbA_\C(\ttb), 
\qquad \text{ $x_{k} \mapsto \bFF_{\uii, k}$ for any $ k\in [1,r]$.} 
\end{align*}
We denote by $\tobbA(\ttb)$ (resp.\ $\tobbA_\C(\ttb)$) the localization of $\obbA(\ttb)$ (resp.\ $\obbA_\C(\ttb)$) at the frozen variables 
\begin{align*} 
\rmF_{t} := \mD(\frakc_t) \qquad \text{for all $t\in \sfK(\frakC)^{\mathrm{fr}}$},
\end{align*} 
and set 
$\tpP(\uii)$ to be the localized ring of $\pP(\uii)$ by the elements $\psi_\uii^{-1}(\rmF_{t})$ for all $t\in \sfK(\frakC)^{\mathrm{fr}}$. 
Using the same notation, we have 
\begin{align} \label{Eq: iso PA}
\psi_{\uii} : \tpP( \uii ) \buildrel \sim \over \longrightarrow \tobbA_\C(\ttb), 
\qquad \text{ $x_{ k } \mapsto \bFF_{\uii, k}$ for any $ k\in [1, r]$.} 
\end{align}
This induces an isomorphism of algebras
\[
\cA(\seed(\frakC)) \simeq \tobbA_\C(\ttb) \simeq \tpP( \uii),
\]
where  $\cA(\seed(\frakC))$ denotes the cluster algebra over $\C$ in Definition \ref{Def: cluster alg}.

\begin{example}\label{ex: ibox seed}
We keep all notations appearing in Example \ref{ex: ibox exchange matrix}.
\bnum
\item \label{ex: ibox seed (i)}
The exchange matrix $B(\frakC_\calR)$ is given in Example \ref{ex: ibox exchange matrix} \eqref{ex: ibox exchange matrix (i)} and the cluster $\mD(\frakC_\calR)  = (\mD(\frakc_t'))_{t\in [1,6]} $ is computed by using the T-systems \eqref{Def: minors}  as follows:
\bna
\item $\mD(\frakc_1') = \mD[1] = \bFF_{1}$, 
\item $\mD(\frakc_2') = \mD[2] = \bFF_{2}$,
\item $\mD(\frakc_3') = \mD[3] = \bFF_{3}$,
\item $\mD(\frakc_4') = \mD[1,4] = \mD[1]\mD[4] - \mD[2]\mD[3] = \bFF_{1}\bFF_{4} - \bFF_{2}\bFF_{3}$,
\item $\mD(\frakc_5') = \mD[1,5] = (\mD[1,4]\mD[4,5] - \mD[2]\mD[3])\mD[4]^{-1} = 
\bFF_{1}\bFF_{4} \bFF_{5} - \bFF_{2}\bFF_{3}\bFF_{5} - \bFF_{1}$,
\item $\mD(\frakc_6') = \mD[3,6] = \mD[3]\mD[6] - \mD[4,5] = \bFF_{3}\bFF_{6} - \bFF_{4}\bFF_{5}+1$,
\ee
where $\mD[4,5] = \bFF_4 \bFF_5 - 1$.

\item \label{ex: ibox seed (ii)}
The exchange matrix $B(\frakC)$ is given in Example \ref{ex: ibox exchange matrix} \eqref{ex: ibox exchange matrix (ii)} and the cluster $\mD(\frakC) = (\mD(\frakc_t))_{t\in [1,6]}$ is computed  by using the T-systems \eqref{Def: minors} and (i) as follows:
\bna
\item $\mD(\frakc_1) = \mD[3] = \bFF_{3}$, 
\item $\mD(\frakc_2) = \mD[2] = \bFF_{2}$,
\item $\mD(\frakc_3) = \mD[4] = \bFF_{4}$,
\item $\mD(\frakc_4) = \mD[1,4] = \bFF_{1}\bFF_{4} - \bFF_{2}\bFF_{3}$,
\item $\mD(\frakc_5) = \mD[1,5]  = 
\bFF_{1}\bFF_{4} \bFF_{5} - \bFF_{2}\bFF_{3}\bFF_{5} - \bFF_{1}$,
\item $\mD(\frakc_6) = \mD[3,6] = \bFF_{3}\bFF_{6} - \bFF_{4}\bFF_{5}+1$.
\ee
Since we have 
$$
\frakC=\calB_1( \calB_3 (\calB_2(\calB_1(\frakC_{\calR})))),
$$
 the exchange matrix $B(\frakC)$ can be obtained from $ B(\frakC_\calR)$ by
\[
B(\frakC) = \sigma_{1,2} \circ \mu_3 \circ \sigma_{2,3} \circ \sigma_{1,2}
\bigl(B(\frakC_{\calR})\bigr).
\]
The cluster variables $\mD(\frakC)$ can also be obtained by applying the same operations 
$\sigma_{1,2}, \sigma_{2,3}, \mu_3, \sigma_{1,2}$ to $\mD(\frakC_\calR)$ successively.
We have 
\begin{align*}
\mD(\frakc_1) = \mD(\frakc_3') \qtq \mD(\frakc_t) = \mD(\frakc_t')\quad \text{ for $t\in \{2,4,5,6 \},$}   
\end{align*}
$\mD(\frakc_3)$ can be obtained by mutating at the vertex corresponding to $\mD(\frakc_1')$.

\ee
\end{example}

\vskip 2em

\section{Double inductive weaves} \label{Sec: weave}

The Demazure weaves are combinatorial objects which provide initial seeds for the coordinate ring of a braid variety. In this section, we recall the notation for double inductive weaves following \cite{CGGLSS25} for our purposes.

Recall the half twist $\Delta = \ttb(w_\circ)$ given in \eqref{Eq: D w0}.
The \emph{Demazure product} $\delta: \Br^+ \to \weyl$ is the map defined inductively by $\delta(\sigma_i) = s_i$ $(i\in I)$ and, for any $\ttb\in \Br^+$,  
\begin{align} \label{Eq: Dp}
\delta(\ttb \sigma_i) = 
\begin{cases}
\delta(\ttb) s_i & \text{if } \ell(\delta(\ttb)s_i) > \ell(\delta(\ttb)), \\
\delta(\ttb) & \text{if } \ell(\delta(\ttb)s_i) < \ell(\delta(\ttb)).
\end{cases}
\end{align}

\subsection{Demazure weaves} \label{Sec: DW}

A Demazure weave $\fW \subset \R^2$ is a planar graph where each edge is labeled by a simple root index $i \in I$. The vertices of the graph fall into specific types as depicted in Figure \ref{fig:vertex_type}: 6-valent vertices, 4-valent vertices, and 3-valent vertices. We also call the three diagrams in Figure \ref{fig:vertex_type} \emph{weave moves}.
By convention, we assume that all weaves are oriented downwards, and use the notation $\fW: \uii \to \ujj$ to specify that a weave $\fW$ starts from the top with $\uii$ and ends at the bottom with $\ujj$.

\begin{figure}[htbp]
  \begin{tikzpicture}[scale=1]
	\draw [ultra thick] 
         (0,0) -- (0,-1) (0,0) -- (-1,1) (0,0) -- (1,1);
         \draw [ultra thick, color=purple] 
         (0,0) -- (0,1) (0,0) -- (-1,-1) (0,0) -- (1,-1);
         \node[above] at (-1,1) {$i$};
         \node[above] at (0,1) {\color{purple}$j$};
         \node[above] at (1,1) {$i$};
  \end{tikzpicture}
   \qquad \qquad
  \begin{tikzpicture}[scale=1]
	\draw [ultra thick] 
         (-1,1) -- (1,-1);
         \draw [ultra thick, color=teal] 
         (1,1) -- (-1,-1);
         \node[above] at (-1,1) {$i$};
         \node[above] at (1,1) {\color{teal}$k$};
  \end{tikzpicture}  
\qquad \qquad
  \begin{tikzpicture}[scale=1]
	\draw [ultra thick] 
         (0,0) -- (0,-1) (0,0) -- (-1,1) (0,0) -- (1,1);         
         \node[above] at (-1,1) {$i$};
         \node[above] at (1,1) {$i$};
  \end{tikzpicture}  
  \caption{Three types of vertices in weaves, where $j$ is adjacent to $i$ in the Dynkin diagram of $\g$ while $k$ is not.}
  \label{fig:vertex_type}
\end{figure}
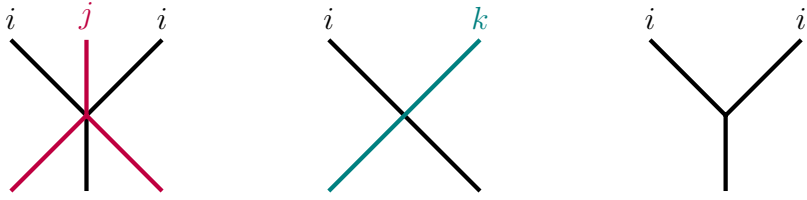

Let $\uii \in \Seq(\ttb) $ for $ \ttb \in \Br^+$ and let $\ujj \in \Seq (\delta(\ttb))$.
As horizontal slices of a weave $\fW: \uii \to \ujj$ can be viewed as expression sequences of elements in $\Br^+$, one can understand that a weave $\fW$ represents a history of the weave moves applied to $\uii$ at the top. Note that the weave moves are compatible with the Demazure product (see \eqref{Eq: Dp}). When we would not emphasize the explicit description of $\ujj$,  
by abuse of notation, we may replace $\ujj$ with $\delta(\uii) := \delta(\ttb)$ in the weave, i.e., $\fW: \uii \to \delta(\uii)$. Throughout this paper, a weave $\fW: \uii \to \delta(\uii)$ will be referred to as a Demazure weave on $\uii$.

Given a Demazure weave $\fW$, we associate to it a quiver $Q(\fW)$. The vertex set of $Q(\fW)$ is identified with the set $\fW_3$ of trivalent vertices in $\fW$, and 
the arrows, as well as the set $\mathcal{F}\subset \fW_3$ of frozen vertices, are determined by the local intersections of \emph{Lusztig cycles} at $\fW_3$ (see \cite[Definition~4.26]{CGGLSS25} for details).

An \emph{indexed} weave is a weave with a bijection from the set $\fW_3$ of trivalent vertices to $\{ 1,2,\ldots,m \}$, where $m:= |\fW_3|$ is the total number of trivalent vertices. If $\fW$ is indexed, then the quiver $Q(\fW)$ can be understood as a labeled quiver whose vertices are in bijection with $\{ 1,2,\ldots,m \}$.

Two weaves $\fW_1, \fW_2: \uii \to \ujj$ are said to be \emph{equivalent} if they are related by a sequence of the following local moves:
\bna
    \item isotopies of strands corresponding to braid relations in $\Br^+$,
    \item local moves involving $i j i j \to j i j j \to j i j$ and $i j i j \to i i j i \to i j i \to j i j$ for adjacent $i$ and $j$,
    \item local moves involving $i k i \to i i k \to i k \to k i$ and $i k i \to k i i \to k  i $ for non-adjacent $i$ and $k$. 
\ee
Note that these local moves involve at most one trivalent vertex. Hence, by transporting the
indexing, these moves also define an equivalence relation on the set of indexed weaves.

Equivalent (indexed) weaves induce the same (labeled) quivers and cluster variables (see, e.g., \cite[Section 4.2 and Proposition 6.6]{CGGLSS25}). By abuse of notation, we will often denote by $\fW$ the equivalence class of an (indexed) weave.

\subsection{Double inductive weaves} \label{Sec: diw}

A \emph{double string} is a sequence of pairs
\[
\bfs=(i_1 \bfX_1, i_2 \bfX_2, \ldots, i_m \bfX_m),
\]
where $i_t\in I$ and $\bfX_t\in \{\bfL, \bfR\}$.
Each double string $\bfs$ determines a sequence
$\uii_0, \uii_1, \ldots, \uii_m$ recursively,
where $ i_t  \bfX_t$ specifies whether $ i_t $ is added on the left or the right of $\uii_{t-1}$ according to $\bfX_t$. More precisely, set $\uii_0 := \emptyset$ and for $t=1,\ldots,m$, define inductively
\[
\uii_t :=
\begin{cases}
(i_t)* \uii_{t-1} & \text{if } \bfX_t=\bfL,\\
\uii_{t-1}* (i_t) & \text{if } \bfX_t=\bfR,
\end{cases}
\]
where $\ujj_1*\ujj_2$ denotes the concatenation of $\ujj_1$ and $\ujj_2$.
We denote by $\uii(\bfs)$ the final sequence $\uii_m$, which is called the \emph{$\uii$-sequence} of $\bfs$.
If adding the index $i_t$ increases the length of the
Demazure product, i.e., \( \ell(\delta(\uii_t))=\ell(\delta(\uii_{t-1}))+1 \),
we put the symbol ``$+$'' as a superscript on $i_t\bfX_t$ for more clarity.

For any double string $\bfs$, we define the corresponding \emph{double inductive weave} 
$$
\fW(\bfs): \uii(\bfs) \to \delta(\uii(\bfs))
$$  
as follows.
We start with the empty weave.  At the $t$-th stage, 
we add one strand to the current weave according to $i_t\bfX_t$
and apply a sequence of weave moves to the bottom of the diagram:
\begin{enumerate}
    \item \textbf{Length-additive step ($i_t \bfX_t^+$):} 
    If $\ell(\delta(\uii_t)) = \ell(\delta(\uii_{t-1})) + 1$, then we add a strand of color $i_t$ on the side specified by $\bfX_t$. This extends the sequence without introducing new trivalent vertices.
    
    \item \textbf{Non-additive step ($i_t \bfX_t$):} 
    If $\ell(\delta(\uii_t)) = \ell(\delta(\uii_{t-1}))$, then the Demazure product already has a descent at $i_t$ on the side $\bfX_t$. 
    We add a strand of color $i_t$ to the side of the current weave specified by $\bfX_t$, and merge the new strand with the current weave to make a 3-valent vertex on the side $\bfX_t$ using 6-valent and 4-valent moves if necessary.  
\end{enumerate}
Since the trivalent vertices occur only in the non-additive steps, 
we order the set $\fW_3 (\bfs)$ of trivalent vertices from top to bottom according to the non-additive steps. This gives an indexing of the vertex set $\fW_3 (\bfs)$ by $\{1,2,\ldots,m\}$, where $m:=|\fW_3(\bfs)|$. In this sense, we consider $Q(\fW(\bfs))$ as a labeled quiver.
We set 
$$
Q(\bfs) := Q(\fW(\bfs))
$$ 
to be the quiver associated with the double string $ \bfs$.
Note that the resulting weave $\fW(\bfs): \uii(\bfs)\to\delta(\uii(\bfs))$
is well-defined up to weave equivalence.

A double inductive weave is called a \emph{right inductive weave} if the
associated double string is
$
(i_1\bfR, i_2\bfR, \ldots,i_m\bfR).
$
In this case, putting $ \uii = (i_1, i_2, \ldots, i_m)$, we write $\rind{\uii}$ for this right inductive weave for brevity, and call it the right inductive weave associated with $\uii$.

Let $\rxw  \in \Seq(\Delta)$ be an expression sequence of the half twist $\Delta$ defined in  \eqref{Eq: D w0}. 
By definition, for any $\uii \in \Seq(\ttb)$, we have $\delta(\rxw  \uii) = w_\circ $.  
Then the quiver of the right inductive weave $\rind{\rxw \uii}$ is described recursively as follows.

\begin{proposition}[{\cite[Proposition 4.44]{CGGLSS25}}] \label{prop: right inductive weave quiver}
Let $\rxw  \in \Seq(\Delta)$, $i \in I$ and $\ttb \in \Br^+$. We take $\uii \in \Seq(\ttb)$ and set  $ \uii' := \uii*(i) $ to be the concatenation of $\uii$ and $(i)$.
For $j\in I$, we denote by $f(j)$  the unique frozen vertex associated with $j$ in the quiver $Q({\rind{\rxw \uii}})$ (if such a vertex exists). The quiver $Q({\rind{\rxw  \uii' }})$ is obtained from $Q({\rind{\rxw  \uii}})$ by the following procedure:
\bna
    \item Make the vertex $f(i)$ mutable and add a new frozen vertex $f'(i)$, together with an arrow $f(i) \to f'(i)$.
    \item If $j \in I$ is adjacent to $i$ in the Dynkin diagram and the frozen vertex $f(j)$ was added after $f(i)$ in the inductive process, add an arrow $f(j) \to f(i)$.
\ee
\end{proposition}

\begin{remark}
In \cite{CGGLSS25}, the quiver $Q(\rind{\rxw \uii})$ may include arrows between frozen vertices. But in the convention of this paper, there are no arrows between frozen vertices. Thus we drop all arrows of weight $1/2$ between frozen vertices in \cite[Proposition 4.44]{CGGLSS25}, and write Proposition \ref{prop: right inductive weave quiver} by modifying \cite[Proposition 4.44]{CGGLSS25} slightly for our purposes.  
\end{remark}

\begin{example}\label{ex:initial_variables 1}
Let $\g$ be a simple Lie algebra of type $A_3$ with index set $ I=\{1,2,3\}$.
We take an expression sequence $\rxw \in \Seq(\Delta)$ of the half twist $\Delta$ and set
$$
 \uii:=(2,3,1,2,2,1).
$$
Let $\fW = \rind{\rxw   \uii}$ be the right inductive weave associated with $\rxw  \uii$. Then the quiver $Q(\fW)$  is shown as follows.
$$
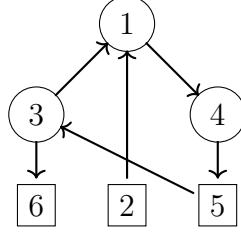
\begin{tikzpicture}[scale=0.6]
    \coordinate (1) at (0,0);
    \coordinate (2) at (0,-4);
    \coordinate (3) at (-2,-2);
    \coordinate (4) at (2,-2);
    \coordinate (5) at (2,-4);
    \coordinate (6) at (-2,-4);

    \node[circle, draw=black, minimum size=0.5cm] at (1) {1};
    \node[rectangle, draw=black, minimum size=0.5cm] at (2) {2};
    \node[circle, draw=black, minimum size=0.5cm] at (3) {3};
    \node[circle, draw=black, minimum size=0.5cm] at (4) {4};
    \node[rectangle, draw=black,  minimum size=0.5cm] at (5) {5};
    \node[rectangle, draw=black, minimum size=0.5cm] at (6) {6};

    \draw[->,  thick, shorten >=10pt, shorten <=10pt] (2.north) -- (1.south);            
    \draw[->,  thick, shorten >=10pt, shorten <=10pt] (3.north east) -- (1.south west);  
    \draw[->,  thick, shorten >=10pt, shorten <=10pt] (1.south east) -- (4.north west);  
    \draw[->,  thick, shorten >=10pt, shorten <=10pt] (4.north east) -- (5.south west);  
    \draw[->,  thick, shorten >=10pt, shorten <=10pt] (5.west) -- (3.north east);        
    \draw[->,  thick, shorten >=10pt, shorten <=10pt] (3.north) -- (6.south west);       
\end{tikzpicture}
$$
Note that this quiver coincides with $Q(\frakC_{\calR})^{\mathrm{op}}$ in Example~\ref{ex: ibox exchange matrix} \eqref{ex: ibox exchange matrix (i)}.
\end{example}

\vskip 2em 

\section{Braid varieties}\label{sec: braid varieties}

This section reviews the cluster algebra structures on the coordinate rings of braid varieties and the connection to the (half decorated) double Bott-Samelson varieties, following \cite{CGGLSS25}. This cluster algebra structure is combinatorially encoded by Demazure weaves. In particular, we will focus on the cluster algebra structures arising from double inductive weaves associated with double strings.

\subsection{Braid varieties}


Let $\sfG$ be a simply connected simple algebraic group over $\C$ of \emph{finite $ADE$ type}, and denote by $\g$ the corresponding Lie algebra.
 We fix a Borel subgroup $\sfB$, denote by $\sfB_-$ the opposite one, and set $\sfT := \sfB \cap \sfB_-$. Let $\weyl = \ang{s_i \mid i\in I}$ be the Weyl group associated with $\g$ with the index set $I$. 
Elements of the \emph{flag variety} $\fl := \sfG/\sfB$ are considered as Borel subgroups of $\sfG$ and the identity coset of $\fl$ is called the standard flag corresponding to the Borel subgroup $\sfB$. There is the Bruhat decomposition $\sfG = \bigsqcup_{w \in  \weyl } \sfB w \sfB$, where we identify $\weyl$ with the quotient $N_\sfG(\sfT) / \sfT$.
We say that a pair of flags $(x\sfB, y\sfB) \in \fl \times \fl$ is in \emph{relative position} $w$ if 
$$
x^{-1}y \in \sfB w \sfB,
$$ 
which is denoted by $x\sfB \xrightarrow{w} y\sfB$.

\begin{definition}[{\cite[Section 3.3]{CGGLSS25}}]
Let $\ttb \in \Br^+$ and $\uii = (i_1, \dots, i_{r}) \in \Seq(\ttb)$. 
The \emph{braid variety} $X(\uii)$ is the subvariety of $\fl^{r+1}$ given by
\[
X(\uii) := \left\{ (\sfB_0, \dots, \sfB_{r}) \in \fl^{r+1} \;\middle|\; 
\begin{aligned} 
&\sfB_0 = \sfB, \quad \sfB_{r} =  \delta (\ttb)\sfB, \\
&\sfB_{t-1} \xrightarrow{s_{i_t}} \sfB_{t} \quad \text{for all } 1 \le t \le r
\end{aligned} 
\right\},
\]
where the Demazure product $ \delta (\ttb)$ is understood as an element in $N_\sfG(\sfT)/\sfT$ via the isomorphism $ \weyl \cong N_\sfG(\sfT)/\sfT$.
\end{definition}

By \cite[Lemma 3.1]{CGGLSS25}, different choices of expression sequences of $\ttb$ yield isomorphic varieties.

Following \cite[Section 3.4]{CGGLSS25}, we fix a pinning of the group $\sfG$. This choice determines a homomorphism $\phi_i: \mathrm{SL}_2(\C) \to \sfG$ for each $i \in I$. 
We define the matrix $B_i(z)$ for a parameter $z \in \C$ by
\begin{align}\label{eq: B_i(z)}
B_i(z) := \phi_i  \begin{pmatrix} z & -1 \\ 1 & 0 \end{pmatrix}  = x_i(z) \dot{s}_{i} \qquad 
\text{ for $i\in I$,}
\end{align}
where $x_i(z)$ and $\dot{s}_{i}$ are defined as 
\[
x_i(z) := \phi_i \begin{pmatrix} 1 & z \\ 0 & 1 \end{pmatrix}  \qtq
\dot{s}_{i} := \phi_i   \begin{pmatrix} 0 & -1 \\ 1 & 0 \end{pmatrix}.
\]
Note that  $\dot{s}_{ i } \in N_{\sfG}(T)$ is a representative of the simple reflection $s_i \in \weyl \cong N_{\sfG}( \sfT )/ \sfT $.
For an expression sequence $\uii = (i_{1}, \dots, i_{r}) \in \Seq(\ttb)$ and a sequence of parameters $z = (z_1, \dots, z_r) \in \C^r$, the element $B_{\uii}(z) \in \sfG$ is defined as the ordered product:
\begin{align} \label{Eq: Biz}
B_{\uii}(z) := B_{i_{1}}(z_1)B_{i_{2}}(z_2)\cdots B_{i_{r}}(z_r) \in \sfG.
\end{align}
By \cite[Corollary 3.7]{CGGLSS25}, the braid variety $X(\uii)$ for $\uii = (i_1, \dots, i_r)$ can be realized as an affine subvariety of $\C^r$ via the isomorphism
\begin{align} \label{Eq: Br coord}
X(\uii) \cong \left\{ (z_1, \dots, z_r) \in \C^{r} \;\middle|\; \delta(\ttb)^{-1}B_{\uii}(z) \in \sfB \right\}.
\end{align}
We now understand $z_j \in \C[X(\uii)]$ as the $j$-th coordinate function restricted from the ambient space $\C^r$ under the isomorphism \eqref{Eq: Br coord}. We refer to the tuple $z = (z_1, \ldots, z_r)$ as the \emph{coordinates} on the braid variety $X(\uii)$.

\subsection{Cluster structure} \label{Sec: cs for bv}

Recall Demazure weaves and double strings in Section \ref{Sec: weave}. 

Let $\fW : \uii \to \delta(\uii) $ be a Demazure weave for $\uii = (i_1, i_2, \dots, i_r) \in \Seq(\ttb)$ and 
assign coordinates $z_1, \ldots, z_r$ to the top strands of the weave $\fW$ from left to right. 
Following the inductive procedure of \cite[Section~5.2]{CGGLSS25}, we associate a regular function $A_k(\fW) \in \C[X(\uii)]$ to each $k \in \fW_3$, which is expressed as a polynomial in $z_1, \ldots, z_r$ (see \cite[Section~7.2]{CGGLSS25}). 
We denote the collection of these functions by 
\begin{align} \label{Eq: cv}
\calA(\fW) := \{ A_t(\fW) \}_{t\in [1,m]},
\end{align}
where $m = |\fW_3|$. 
The set $ \calA(\fW)$ together with the labeled quiver $Q(\fW)$ gives an initial seed for $\C[X(\uii)]$.

\begin{theorem}[{\cite[Theorem 1.1]{CGGLSS25}}]\label{thm: CGGLSS cluster structure}
Let $\uii \in \Seq(\ttb)$ and $\fW: \uii \to \delta(\uii)$ be a Demazure weave for $\uii$. 
Then the coordinate ring $\C[X(\uii)]$ of the braid variety $X(\uii)$ has a cluster algebra structure with an initial seed 
\[
\Sigma(\fW) := \bigl(\calA(\fW), Q(\fW) \bigr).
\]
\end{theorem}

The double inductive weaves arising from double strings have enjoyable properties.  For any two double strings $\bfs_1$ and $\bfs_2$ with $\uii(\bfs_1)=\uii(\bfs_2)$, they are connected by the following two types of local moves:
\begin{equation} \label{Eq: moves for strings} 
\begin{aligned} 
(j_1 \bfL, j_2 \bfX_2, \dots) &\longleftrightarrow (j_1 \bfR, j_2 \bfX_2, \dots), \\
(\dots, i \bfL, j \bfR, \dots) &\longleftrightarrow (\dots, j \bfR, i \bfL, \dots).
\end{aligned}
\end{equation}
From the proof of \cite[Theorem 6.8]{CGGLSS25}, we have the following theorem.

\begin{theorem}[{The proof of \cite[Theorem 6.8]{CGGLSS25}}] \label{thm:cases}
Let $\bfs_1$ and $\bfs_2$ be two double strings such that $\uii(\bfs_1)=\uii(\bfs_2)$, differing by one of the following moves: 
\begin{description}
  \item[Case~0] $(j_1 \bfL^+, \dots) \longleftrightarrow (j_1 \bfR^+, \dots)$,
  \item[Case~1] $(\dots, i \bfL^+, j \bfR^+, \dots) \longleftrightarrow (\dots, j \bfR^+, i \bfL^+, \dots)$,
  \item[Case~2] $(\dots, i \bfL^+, j \bfR, \dots) \longleftrightarrow (\dots, j \bfR^+,   i \bfL, \dots)$,
  \item[Case~3] $(\dots, i \bfL^+, j \bfR, \dots) \longleftrightarrow (\dots, j \bfR, i \bfL^+, \dots)$,
  \item[Case~4] $(\dots, i \bfL, j \bfR, \dots) \longleftrightarrow (\dots, j \bfR, i \bfL, \dots)$, with $\ell(s_{i} u s_{j})=\ell(u)-2$,
  \item[Case~5] $(\dots, i \bfL, j \bfR, \dots) \longleftrightarrow (\dots, j \bfR, i \bfL, \dots)$, with $\ell(s_{i} u s_{j})=\ell(u)$,
\end{description}
where $u$ is the Demazure product of the prefix sequence of $\uii(\bfs_1)$ preceding the move.
For each $p \in \{1, 2\}$, we set $\fW_p := \fW(\bfs_p)$ and
$$
\Sigma_p := \Sigma (\fW_p) = (\calA_p := \{ A_t(\fW_p) \}_{t\in[1,m]}, Q_p := Q(\fW_p)\},
$$
where $m=|(\fW_p)_3|$. 
Then the relationship between $\Sigma_1$ and $\Sigma_2$ is as follows:

\begin{enumerate}
\item \textbf{Cases~0--3 (Remaining unchanged):} 
The labeled quivers are identical and the cluster variables coincide:
\[
Q_2 = Q_1, \qquad \calA_2=\calA_1.  
\]
These moves involve length-additive steps that do not introduce new mutable trivalent vertices.

\item \textbf{Case~4 (Relabeling):} 
Let $a$ and $a+1$ be the indices of the trivalent vertices corresponding to the steps $i \bfL$ and $j \bfR$ (or $j \bfR$ and $i \bfL$). 
The quiver $Q_2$ is obtained from $Q_1$ by swapping the labels of vertices $a$ and $a+1$ and
the cluster variables are permuted accordingly:
\[
Q_2 = \sigma_{a,a+1}(Q_1), \qquad \calA_2 = \sigma_{a, a+1}(\calA_1).
\]

\item \textbf{Case~5 (Mutation):} 
Let $a$ be the index of the trivalent vertex corresponding to the step $i \bfL$.
The quiver $Q_2$ is obtained from $Q_1$ by mutation at vertex $a$ and
the set of cluster variables $\calA_2$ is obtained from $\calA_1$ by cluster mutation at $a$. That is, $A_t(\fW_2) = A_t(\fW_1)$ for $t \ne a$, and the variable $A_a(\fW_2)$ is determined by the exchange relation:
\[
Q_2 = \mu_{a}(Q_1), \quad \calA_2 = \mu_a(\calA_1).
\]
\end{enumerate}
\end{theorem}

\begin{remark} \label{Rmk: cases}
In the case where $u= w_\circ$ in Theorem \ref{thm:cases}, Case~2 and Case~5 occur if and only if
$j = i^*$, whereas Case~4 occurs if $j \neq i^*$.
\end{remark}

We now recall the \emph{double Bott-Samelson variety} following \cite{SW24, CGGLSS25}. 
For any expression sequence $\uii = (i_1,i_2, \dots, i_r) \in \Seq(\ttb)$,  the (half decorated) double Bott-Samelson variety $\Conf(\uii)$ is defined as the subvariety of $\C^r$:
\[
\Conf(\uii) := \{(z_1, \dots, z_{r}) \in \C^r \mid B_{\uii}(z) \in \sfB_{-} \cdot\sfB\},
\]
where $B_{\uii}(z)$ is given in \eqref{Eq: Biz}.
The variety $\Conf(\uii)$ is smooth and affine (\cite[Section 2.4]{SW24}). 
Similar to the case of braid varieties, different expression sequences $\uii$ for the same braid element $\ttb$ yield isomorphic varieties.

We write $(w,z)$ for the coordinates of the braid variety $X(\rxw \uii)$, where $w = (w_1, \ldots, w_{ \ell(\rxw)})$ (resp.\ $z = (z_1, \ldots, z_r)$) corresponds to the letters of $\rxw$ (resp.\ $\uii$), and write $z$ for the coordinates of the double Bott-Samelson variety $\Conf(\uii)$. By \cite[Lemma 3.16]{CGGLSS25}, there exists an isomorphism between the braid variety and the double Bott-Samelson variety
\begin{align} \label{Eq: X=Conf}
\zeta_\uii : X(\rxw \uii) \buildrel \sim \over \longrightarrow \Conf(\uii)
\end{align}
defined by sending the coordinates $(w,z)$ to the coordinates $z$.

Recall the right inductive weave $\rind{\rxw\uii}$ in Section \ref{Sec: diw}. 
Under the isomorphism $\zeta_\uii$ in \eqref{Eq: X=Conf}, the cluster variables  
$\calA (\rind{\rxw\uii})$ (see \eqref{Eq: cv}) can be described in terms of generalized principal minors.

\begin{proposition}[{\cite[Proposition 5.20 and Section 7.2]{CGGLSS25}}] \label{prop: conf variables}
Let $\uii=(i_1,\dots,i_r)\in\Seq(\ttb)$. 
We identify $X(\rxw \uii)$ and $\Conf(\uii)$ via the isomorphism $\zeta_\uii$ in \eqref{Eq: X=Conf}. 
 For each $1 \le t \le r$, the cluster variables of $X(\rxw \uii)$ are given by
\[
A_t(\rind{\rxw \uii}) = \PM_{i_t}(B_{i_{1}}(z_1)\cdots B_{i_{t}}(z_t)),
\]
where $\PM_{i}$ is the generalized principal minor associated to the fundamental weight $\varpi_{i}$. 
Furthermore, the coordinate ring of the double Bott-Samelson variety is obtained from localizing the polynomial ring 
\[
\C\bigl[\Conf(\uii)\bigr] \cong \C \bigl[z_1,\dots,z_{r}\bigr] \bigl[f_i(z)^{-1} \mid i \in I \bigr],
\]
where  $f_i (z) = \PM_{i}(B_{\uii}(z))$ for $i\in I$.
\end{proposition}

Thanks to Proposition \ref{prop: conf variables}, for any $\uii = (i_1, i_2, \ldots, i_r) \in \Seq(\ttb)$, we define 
\begin{align*} 
\calB (\uii) := \C \bigl[z_{ 1},z_{ 2}, \dots,z_{ r}\bigr] \bigl[f_i (z)^{-1} \mid i \in I \bigr].
\end{align*}
Throughout the remainder of the paper, we identify the three algebras $\calB (\uii)$, $\C\bigl[\Conf(\uii)\bigr]$ and $\C\bigl[X(\rxw \uii)\bigr] $ via the isomorphisms 
\begin{align} \label{Eq: 3 iso}
\calB (\uii) \cong \C\bigl[\Conf(\uii)\bigr] \cong  \C\bigl[X(\rxw \uii)\bigr], 
\end{align}
where the first isomorphism is given in Proposition \ref{prop: conf variables} and the second is given by \eqref{Eq: X=Conf}.

\begin{example}\label{ex:initial_variables}
We continue with Example \ref{ex:initial_variables 1}. The standard realization $\sfG=\SL_4(\C)$ and the standard pinning are given as follows: for $i \in\{1,2,3\}$,
\[
\phi_i\!\left(\begin{pmatrix} a&b\\ c&d \end{pmatrix}\right)
=
\begin{pmatrix}
I_{i-1} & 0 & 0\\
0 & \begin{pmatrix} a&b\\ c&d \end{pmatrix} & 0\\
0 & 0 & I_{3-i}
\end{pmatrix}.
\]
Then we have
\[
B_i(z)=\phi_i\!\left(\begin{pmatrix} z&-1\\ 1&0 \end{pmatrix}\right),
\]
and $\triangle_{i}(x)$ is the leading principal $i\times i$ minor of $x\in\SL_4(\C)$.
By a direct computation, Proposition~\ref{prop: conf variables} yields the cluster variables $A_t := A_t(\rind{\rxw \uii})$ as follows: the three frozen variables are  
\[
{A}_2=z_2,\qquad 
{A}_5 = z_1z_4z_5 - z_2z_3z_5 - z_1 ,\qquad
{A}_6 = z_3z_6 - z_4z_5 + 1,
\]
and the three mutable variables are
\[
{A}_1=z_1,\qquad 
{A}_3=z_3, \qquad
{A}_4 = z_1z_4 - z_2z_3.
\]
\end{example}

\vskip 2em 

\section{Comparison of two cluster algebra structures}\label{sec: comparison of two cluster algebra structures}

In this section, we establish an explicit isomorphism between the localized bosonic extension $\widetilde{\obbA}(\ttb)$ and the coordinate ring of the braid variety $X(\rxw \uii)$ for $\uii \in \Seq(\ttb)$, and prove that this isomorphism respects the cluster structures. The comparison of these cluster structures relies on the construction of a Demazure weave associated with an admissible chain of $i$-boxes, which is formulated using double strings.

Throughout this section, let $\cmC = (c_{i,j})_{i,j\in I}$ be a Cartan matrix of \emph{finite $ADE$ type}, 
and let $\g$ be the corresponding simple Lie algebra. We denote by $\Br$ the generalized Braid group associated with $\cmC$.

\subsection{Double inductive weaves arising from $i$-boxes} \label{Sec: diw from iBox}
Let $\ttb \in \Br^+$ and take an expression 
$$
\uii :=( i_1, i_2, \ldots,i _r)\in\Seq(\ttb).
$$ 
Let $\frakC = (\frakc_t)_{t \in [1,r]} $ be the admissible chain of $i$-boxes associated with $\uii$ and denote its LR sequence by  $\frakH=(\calH_1,\ldots,\calH_{r-1})$. For simplicity, we set $\calH_0 := \calR$.
For an expression sequence $\rxw = (j_1, j_2, \ldots, j_\ell) \in \Seq(\Delta)$, 
we define 
\begin{align} \label{Eq: s(C)}
\bfs_{\rxw}(\frakC) := (j_1 \bfR^+, j_2 \bfR^+, \ldots, j_\ell \bfR^+, h_1 \bfX_1, h_2 \bfX_2, \ldots, h_r \bfX_r),
\end{align}
where
\begin{align*}
h_t := 
\begin{cases}
\bcol(\frakc_t)^* & \text{if } \calH_{t-1} = \calL, \\
\bcol(\frakc_t) & \text{if } \calH_{t-1} = \calR, 
\end{cases}
\qtq 
\bfX_t :=
\begin{cases}
\bfL & \text{if } \calH_{t-1} = \calL,\\
\bfR & \text{if } \calH_{t-1} = \calR, 
\end{cases}
\end{align*}
(see \eqref{Eq:istar} for $i^*$) and set
\begin{align*}
\uii^\rmL(\frakC)
&:= (h_t)_{\bfX_t=\bfL}
\quad\text{(taken in decreasing order of $t$)},\\
\uii^\rmR(\frakC)
&:= (h_t)_{\bfX_t=\bfR}
\quad\text{(taken in increasing order of $t$)}.
\end{align*} 
Then the $\uii$-sequence of the double string $\bfs_{\rxw}(\frakC)$ is given in terms of $\uii^\rmL(\frakC)$ and $\uii^\rmR(\frakC)$ as follows:
\[
\uii_\rxw (\frakC):= \uii(\bfs_{\rxw}(\frakC))
= \uii^\rmL(\frakC)\,\rxw \,\uii^\rmR(\frakC).
\]
We define 
\begin{align} \label{Eq: fWT}
\fW_{\rxw}^{\rm T}(\frakC) : \rxw\uii \rightarrow \uii_\rxw (\frakC)
\end{align}
to be the Demazure weave consisting of the braid moves, i.e., 4-valent and 6-valent vertices in Figure \ref{fig:vertex_type}, and set $\fW_{\rxw}^{B}(\frakC)$ to be the double inductive weave associated with $\bfs_{\rxw}(\frakC)$ (see Section \ref{Sec: diw}), i.e., 
\begin{align} \label{Eq: fWB}
\fW_{\rxw}^{B}(\frakC) := \fW( \bfs_{\rxw}(\frakC)) :\uii_\rxw (\frakC) \rightarrow \delta ( \uii_\rxw (\frakC)).
\end{align}
Note that $\delta ( \uii_\rxw (\frakC)) = w_\circ$ in the Weyl group $\weyl$. 
We now define  
\begin{align*}
\fW_{\rxw}(\frakC) :=  \fW_{\rxw}^{B}(\frakC) \circ \fW_{\rxw}^{T}(\frakC) : \rxw\uii \rightarrow \delta ( \uii_\rxw (\frakC)),
\end{align*}
where $\fW_2\circ\fW_1$ denotes the vertical concatenation obtained by placing
$\fW_2$ below $\fW_1$. 
Although the weave $\fW_{\rxw}(\frakC)$ itself depends on the choice of $\fW_{\rxw}^{\rm T}(\frakC)$,   all such $\fW_{\rxw}(\frakC)$ belong to the same equivalence class (see Section \ref{Sec: DW}).  Thanks to Proposition \ref{Prop: we seed} below, the seeds arising from $\fW_{\rxw}(\frakC)$ do not depend on the choice of the weave $\fW_{\rxw}^{\rm T}(\frakC)$. 

\begin{proposition} [{\cite[Lemma 4.36 and Lemma 5.15]{CGGLSS25}}] \label{Prop: we seed}
Let $\fW_1,\fW_2 : \bfj \to \delta(\bfj)$ be two equivalent weaves. Then the seeds $\Sigma(\fW_1)$ and $\Sigma(\fW_2)$ coincide.  
\end{proposition}

We call $\bfs_{\rxw}(\frakC)$ (resp.\ $\fW_{\rxw}(\frakC)$) the double string (resp.\ the Demazure weave) associated with $\frakC$.
 We sometimes write $\bfs (\frakC)$ and $\fW (\frakC)$ instead of $\bfs_{\rxw}(\frakC)$ and $\fW_{\rxw}(\frakC)$ respectively when there is no need to emphasize the expression $\rxw$.

\begin{example}\label{ex:double_weave}
We continue with Example~\ref{ex: ibox exchange matrix} \eqref{ex: ibox exchange matrix (ii)}. Let 
$$
\uii=(2,3,1,2,2,1)
$$
and let $\frakC = (\frakc_t)_{t\in [1,6]}$ be the admissible chain associated with $\uii$ determined by the LR sequence 
$\calH = (\calL,\calR,\calL,\calR,\calR)$. Then we have 
\begin{align*}
(\bfX_t)_{1\leq t \leq 6}&=(\bfR,\bfL,\bfR,\bfL,\bfR, \bfR),\\
(h_t)_{1\leq t \leq 6}&=(1,~3^*,~2,~2^*,~2,~1) = (1,~1,~2,~2,~2,~1),
\end{align*}
which implies that
\[
\uii^\rmL(\frakC)=(2^*,3^*) = (2,1) \qtq
\uii^\rmR(\frakC)=(1,2,2,1).
\]
If we take $\rxw=(1,2,1,3,2,1)\in\Seq(\Delta)$, the double string is given as
\begin{align*}
\bfs_\rxw(\frakC) &= (1 \bfR^+,2 \bfR^+,1 \bfR^+,3 \bfR^+,2 \bfR^+,1 \bfR^+,1 \bfR,1 \bfL,2 \bfR,2 \bfL,2 \bfR,1 \bfR),
\end{align*}
and the $\uii$-sequence of the double string is
\[
\uii_\rxw(\frakC) = (\underbrace{2,1}_{\uii^\rmL(\frakC)},\underbrace{1,2,1,3,2,1}_{\rxw},
\underbrace{1,2,2,1}_{\uii^\rmR(\frakC)}).
\]
Figure \ref{fig:weave_quiver} illustrates the weave $\fW_{\rxw}^{B}(\frakC)$ with its Lusztig cycles and its associated quiver $Q(\fW_{\rxw}^{B}(\frakC))$. The quiver $Q(\fW_{\rxw}^{B}(\frakC))$ coincides with the opposite quiver $Q(\frakC)^{\rm{op}}$ of $Q(\frakC)$ in Example~\ref{ex: ibox exchange matrix} \eqref{ex: ibox exchange matrix (ii)}.
\end{example}

\begin{figure}[htbp]
  \begin{tikzpicture}[scale=0.6]
	\draw [ultra thick] 
         (1/2,15) -- (1/2,8) -- (2/2,7.5) -- (2/2,7) -- (1/2,6) -- (1/2,5.5) -- (0,5) -- (0,-0.5)
         (2/2,15) -- (2/2,14) -- (1/2,13)
         (4/2,15) -- (4/2,14) -- (3/2,13) -- (5/2,11.5) -- (5/2,9.5) -- (4/2,9) -- (3/2,8) --(2/2,7.5)
         (7/2,15) -- (7/2,14) -- (6/2,13) -- (6/2,12) -- (5/2,11.5)
         (8/2,15) -- (7/2,14)
         (5/2,9.5) -- (6/2,9) --(6/2,7) -- (5/2,6) -- (5/2,6) -- (5/2,4) -- (4/2,3.5) -- (4/2,1.5) -- (5/2,1) -- (5/2,-0.5)
         (1/2,5.5) -- (2/2,5) -- (3/2,4) -- (4/2,3.5)
         (4/2,1.5) -- (3/2,1) -- (3/2,-0.5)
         (11/2,15) -- (9/2,13) -- (9/2,11) -- (8/2,10) -- (8/2,7) -- (7/2,6) -- (7/2,3) -- (6/2,2) -- (6/2,1) -- (5/2,0);
         \draw [ultra thick, color=purple] 
         (0,15) -- (0,6) -- (1/2,5.5) -- (1/2,-0.5)
         (3/2,15) -- (3/2, 14) -- (2/2,13) -- (2/2,7.5) -- (1/2,7) -- (0,6)
         (2/2,7.5) -- (3/2,7) -- (2/2,6) -- (1/2,5.5)
         (6/2,15) -- (6/2,14) -- (5/2,13) -- (5/2,11.5) -- (4/2,11) -- (4/2,10) -- (5/2,9.5) -- (5/2,7) -- (4/2,6) 
         -- (4/2,3.5) -- (3/2,3) -- (3/2,2) -- (4/2,1.5) -- (4/2,-0.5)
         (5/2,11.5) -- (6/2,11) -- (6/2,10) -- (5/2,9.5)
         (4/2,3.5) -- (5/2,3) -- (5/2,2) -- (4/2,1.5)
         (9/2,15) -- (7/2,13) -- (7/2,11) -- (6/2,10.5)
         (10/2,15) -- (8/2,13) -- (8/2,11) -- (7/2,10) -- (7/2,7) -- (6/2,6) -- (6/2,3) -- (5/2,2.5);
         \draw [ultra thick, color=teal]
         (5/2,15) -- (5/2,14) -- (3/2,12) -- (3/2,9) -- (4/2,8) -- (4/2,7) -- (3/2,6) -- (3/2,5) -- (2/2,4) -- (2/2,-0.5); 
         \foreach \i/\j in {0/2, 1/1, 2/1, 3/2, 4/1, 5/3, 6/2, 7/1,8/1,9/2,10/2,11/1}{
         \node[above] at (\i/2,15) {$\j$};
         }
         \draw[gray, line width=6pt, opacity=0.3] 
         (7/2,14) -- (6/2,13) -- (6/2,12) -- (4/2,11) -- (4/2,10) -- (6/2,9) -- (6/2,7) -- (5/2,6) -- (5/2,4) -- (3/2,3) 
         -- (3/2,2) -- (5/2,1) -- (5/2,0);    
        \draw[pink, line width=6pt, opacity=0.3] 
         (1/2,13) -- (1/2,8) -- (3/2,7) -- (2/2,6) -- (0/2,5) -- (0,-0.5);
        \draw[teal, line width=6pt, opacity=0.3] 
         (6/2,10.5) -- (6/2,10) -- (4/2,9) -- (3/2,8) -- (1/2,7) -- (0/2,6);
        \draw[blue, line width=6pt, opacity=0.3] 
        (0/2,6) -- (2/2,5) -- (3/2,4) -- (5/2,3) -- (5/2,2.5);
       \draw[yellow, line width=6pt, opacity=0.3] 
        (5/2,2.5) --  (5/2,2) --  (3/2,1) -- (3/2,-0.5);
        \draw[purple, line width=6pt, opacity=0.3] 
        (5/2,0) -- (5/2,-0.5);
  \end{tikzpicture}
   \qquad \qquad
 \begin{tikzpicture}[scale=0.7]
    \coordinate (3) at (0,4);
    \coordinate (2) at (0,0);
    \coordinate (1) at (-2, 2);
    \coordinate (4) at (2,2);
    \coordinate (5) at (2,0);
    \coordinate (6) at (-2,0);
    	\node[circle, draw=black, fill=gray!50, minimum size=0.5cm] at (1) {1};
        \node[rectangle, draw=black, fill=pink!50, minimum size=0.5cm] at (2) {2}; 
        \node[circle, draw=black, fill=teal!50, minimum size=0.5cm] at (3) {3}; 
        \node[circle, draw=black, fill=blue!50, minimum size=0.5cm] at (4) {4}; 
        \node[rectangle, draw=black, fill=yellow!50, minimum size=0.5cm] at (5) {5};
        \node[rectangle, draw=black, fill=purple!50, minimum size=0.5cm] at (6) {6}; 
        \node at (0,-7) {};

    \node[circle, draw=black, minimum size=0.5cm] at (3) {3};
    \node[rectangle, draw=black, minimum size=0.5cm] at (2) {2};
    \node[circle, draw=black, minimum size=0.5cm] at (1) {1};
    \node[circle, draw=black, minimum size=0.5cm] at (4) {4};
    \node[rectangle, draw=black,  minimum size=0.5cm] at (5) {5};
    \node[rectangle, draw=black, minimum size=0.5cm] at (6) {6};

    \draw[->,  thick, shorten >=10pt, shorten <=10pt] (3) -- (1);            
    \draw[->,  thick, shorten >=10pt, shorten <=10pt] (5) -- (1);  
    \draw[->,  thick, shorten >=10pt, shorten <=10pt] (3) -- (2);
    \draw[->,  thick, shorten >=10pt, shorten <=10pt] (4) -- (3);
    \draw[->,  thick, shorten >=10pt, shorten <=10pt] (1) -- (4);
    \draw[->,  thick, shorten >=10pt, shorten <=10pt] (2) -- (4);    
    \draw[->,  thick, shorten >=10pt, shorten <=10pt] (4) -- (5);
    \draw[->,  thick, shorten >=10pt, shorten <=10pt] (1) -- (6);    
\end{tikzpicture}   
\caption{The double inductive weave $\fW_{\rxw}^{B}(\frakC)$ associated with the double string 
$\bfs_\rxw(\frakC) = (1 \bfR^+,2 \bfR^+,1 \bfR^+,3 \bfR^+,2 \bfR^+,1 \bfR^+,1 \bfR,1 \bfL,2 \bfR,2 \bfL,2 \bfR,1 \bfR)$ together with its labeled quiver $Q(\fW_{\rxw}^{B}(\frakC))$ in Example \ref{ex:double_weave}.
For each trivalent vertex, the corresponding Lusztig cycle assigns weight either $1$ or $0$ to each edge. 
The thick shaded lines indicate the edges with weight $1$.
Different Lusztig cycles are shown in different colors, which also mark the corresponding vertices in the quiver.}
\label{fig:weave_quiver}
\end{figure}

\subsection{Box moves and weaves}
In this subsection, we will investigate how box moves on an admissible chain $\frakC$ affect the weave $\fW_{\rxw}(\frakC)$.  

We first consider the weave $\fW_{\rxw}^B(\frakC)$ given in \eqref{Eq: fWB}. 
Recall the admissible chain $\frakC_{\calR}$ determined by the LR sequence $(\calR,\calR,\ldots,\calR)$ (see \eqref{Eq: C calR}).

\begin{lemma} \label{lem: RRR and riw}
Let
$\frakC_{\calR}$
be the admissible chain of $i$-boxes associated with $\uii  \in \Seq(\ttb)$ determined by $(\calR, \calR,\ldots,\calR)$, and 
let $\rind{\rxw\uii}$ be the right inductive weave for $\rxw \in \Seq(\Delta)$. Then we have
\[
\fW_{\rxw}^B(\frakC_\calR) = \rind{\rxw\uii}  \qtq
Q(\frakC_{\calR}) = Q(\fW_{\rxw}^B(\frakC_\calR))^{\mathrm{op}}
\]
as labeled quivers.
\end{lemma}
\begin{proof}
The first equality follows from the definition \eqref{Eq: s(C)}, and the second equality follows directly from 
Proposition \ref{prop: right inductive weave quiver} and \eqref{eq: exchange matrix for RRR}.
\end{proof}

The following proposition shows that the quivers and cluster variables associated with $\fW_{\rxw}^B(\frakC)$ transform by the same mutation and permutation operations as the exchange matrices in Proposition~\ref{prop:boxmove_seed}.

\begin{proposition}\label{prop: double string boxmove}
Let $\frakC = (\frakc_t)_{t\in [1,r]}$ be an admissible chain associated with $\uii = (i_1,i_2,\ldots, i_r)$ and let $\frakc_k$ be a movable box.
\bnum
    \item Suppose $k=1$. Let $\rxw=(j_1,j_2,\dots,j_\ell) $ with $j_1= \bcol(\frakc_1)^*$, and set $\rxw':=(j_2,\dots,j_\ell,j_1^*)$. Then the quiver and cluster variables are changed by the box move $ \calB_1$ as follows:
    \begin{equation*}
       Q\bigl( \fW^B_{\rxw'} (\calB_1(\frakC) ) \bigr) =
            \begin{cases}
                \mu_1 \bigl( Q ( \fW^B_{\rxw} (\frakC )  )\bigr)
                & \text{if $\bcol(\frakc_1)=\bcol(\frakc_2)$}, \\[0.3em]
                \sigma_{1,2} \bigl( Q(\fW^B_{\rxw} (\frakC) ) \bigr)
                & \text{if $\bcol(\frakc_1)\neq\bcol(\frakc_2)$},
            \end{cases}
    \end{equation*}
as labeled quivers and
    \begin{equation*}
    \calA\bigl( \fW^B_{\rxw'} (\calB_1(\frakC) \bigr) \bigr) =
    \begin{cases}
    \mu_1\bigl(\calA \bigl(\fW^B_{\rxw} (\frakC) \bigr) \bigr) 
    & \text{if $\bcol(\frakc_1)=\bcol(\frakc_2)$},\\[0.3em]
    \sigma_{1,2}\bigl(\calA\bigl(\fW^B_{\rxw} (\frakC) \bigr)  \bigr)
    & \text{if $\bcol(\frakc_1)\neq\bcol(\frakc_2)$}.
    \end{cases}
    \end{equation*}
    
    \item Suppose $k\ge 2$. For any $\rxw \in \Seq(\Delta)$, the quiver and cluster variables are changed by the 
    box move $ \calB_k$ as follows:
    \begin{equation*}
       Q\bigl(\fW^B_{\rxw} ( \calB_k (\frakC))\bigr) =
            \begin{cases}
                \mu_k\bigl(Q(\fW^B_{\rxw} (\frakC))\bigr)
                & \text{if $\bcol(\frakc_k)=\bcol(\frakc_{k+1})$}, \\[0.3em]
                \sigma_{k,k+1}\bigl(Q(\fW^B_{\rxw} (\frakC))\bigr)
                & \text{if $\bcol(\frakc_k)\neq\bcol(\frakc_{k+1})$},
            \end{cases}
    \end{equation*}
as labeled quivers and
    \begin{equation*}
    \calA\bigl(\fW^B_{\rxw} ( \calB_k (\frakC))\bigr)=
    \begin{cases}
    \mu_k\bigl(\calA(\fW^B_{\rxw} (\frakC))\bigr)
    & \text{if $\bcol(\frakc_k)=\bcol(\frakc_{k+1})$},\\[0.3em]
    \sigma_{k,k+1}\bigl(\calA(\fW^B_{\rxw} (\frakC))\bigr)
    & \text{if $\bcol(\frakc_k)\neq\bcol(\frakc_{k+1})$}.
    \end{cases}
    \end{equation*}
\ee
Here, $\mu_k$ denotes mutation at vertex $k$ and $\sigma_{a,b}$ denotes the permutation swapping the indices $a$ and $b$.
\end{proposition}

\begin{proof}
Let $\rxw = (j_1,j_2, \ldots, j_\ell)$, and denote the LR sequences of $\frakC$ and $\frakC' := \calB_k(\frakC)$ by 
$$
\frakH=(\calH_1,\calH_2,\ldots,\calH_{r-1}) \qtq 
\frakH ' = (\calH_1',\calH_2',\ldots,\calH_{r-1}'),
$$
respectively. We set the double string associated with $\frakC$ as
\begin{align*}
\bfs &:= \bfs_\rxw(\frakC) = (j_1 \bfR^+,\ldots, j_\ell \bfR^+, h_1 \bfX_1,  \ldots, h_r \bfX_r).
\end{align*}

(i) Suppose $k=1$. We assume that $ j_1 = h_1^*$, and write $ \rxw' = (j_2, \ldots, j_\ell, j_1^*)$. We set 
\begin{align*}
\bfs' &:= \bfs_{\rxw'}(\frakC') = (j_2 \bfR^+,\ldots, j_\ell \bfR^+, j_1^* \bfR^+, h_1' \bfX_1', h_2' \bfX_2', \ldots, h_r' \bfX_r').
\end{align*}
Since the box move $\calB_1$ only affects $\calH_1$ and $\calH_1'$,
we may assume $\bfX_2=\bfR$ without loss of generality.
By construction, we have 
\begin{itemize}
\item $\bfX_1 = \bfX_1' = \bfR$ and   $ \bfX_2' = \bfL$, 
\item $h_1' = h_2$ and $h_2' = h_1^*$,
\item $ \bfX_t = \bfX_t'$ and $ h_t = h_t'$ for any $t \ge 3 $,
\end{itemize}
which implies 
$$
\uii(\bfs) = \uii(\bfs').
$$
To compare $\bfs$ and $\bfs'$, we consider the following sequence of moves given in \eqref{Eq: moves for strings}:
\begin{align*}
\bfs=&
(j_1 \bfR^+,j_2 \bfR^+,\ldots, j_\ell \bfR^+,
h_1 \bfR, h_2 \bfR,h_3 \bfX_3,\ldots,  h_r \bfX_r) \\
\underset{\text{\rm Case~0}}{\longleftrightarrow} & 
(j_1 \bfL^+,j_2 \bfR^+,\ldots, j_\ell \bfR^+,
h_1 \bfR, h_2 \bfR,h_3 \bfX_3,\ldots,  h_r \bfX_r) \\
\underset{\text{\rm Case~1}}{\longleftrightarrow} & 
(j_2 \bfR^+,j_1 \bfL^+,\ldots, j_\ell \bfR^+,
h_1 \bfR, h_2 \bfR,h_3 \bfX_3,\ldots,  h_r \bfX_r) \\
\vdots  \quad  & \qquad \qquad \qquad \qquad \qquad \vdots  \\
\underset{\text{\rm Case~1}}{\longleftrightarrow} & 
(j_2 \bfR^+,j_3 \bfR^+,\ldots, j_1 \bfL^+,
h_1 \bfR, h_2 \bfR,h_3 \bfX_3,\ldots,  h_r \bfX_r) \\
\underset{\text{\rm Case~2}}{\longleftrightarrow} & 
(j_2 \bfR^+,j_3 \bfR^+,\ldots, h_1 \bfR^+,
j_1 \bfL, h_2 \bfR,h_3 \bfX_3,\ldots,  h_r \bfX_r)\\
\underset{(*)}{\longleftrightarrow} & 
(  j_2  \bfR^+,  j_3   \bfR^+,\ldots, h_1 \bfR^+,
h_2 \bfR,j_1 \bfL,h_3 \bfX_3,\ldots,  h_r \bfX_r)
= \bfs'.
\end{align*}
As all moves except the last one are of Cases~0--2, such moves do not change the labeled quiver  and 
 the cluster variables by Theorem~\ref{thm:cases}.
Thus, the change in the quiver and cluster variables is determined entirely by the last move $(*)$.
By Remark \ref{Rmk: cases}, we have 
\begin{itemize}
\item $(*)$ is (Case 5) if $h_1 = h_2$, i.e., $\bcol(\frakc_1)=\bcol(\frakc_2)$,
\item $(*)$ is (Case 4) if $h_1 \ne h_2$, i.e., $\bcol(\frakc_1)\ne\bcol(\frakc_2)$.
\end{itemize}
Thus the assertion follows from Theorem~\ref{thm:cases}.

(ii)
 Suppose $k\ge 2$. We set 
\begin{align*}
\bfs' &:= \bfs_{\rxw}(\frakC') = (j_1 \bfR^+,\ldots,  j_\ell \bfR^+, h_1' \bfX_1',  \ldots, h_r' \bfX_r').
\end{align*}
By construction, we have 
\begin{itemize}
\item $\bfX_k \ne \bfX_{k+1}  $,
\item $h_k = h_{k+1}'$ and $h_{k+1} = h_k'$, 
\item $ \bfX_t = \bfX_t'$ and $ h_t = h_t'$ for any $   t   \notin \{ k, k+1 \}$, 
\end{itemize}
which says that $ \uii(\bfs) = \uii(\bfs') $. The double strings $\bfs$ and $\bfs'$ 
differ by one move:
\begin{align*}
\bfs=&
(j_1 \bfR^+,\ldots, j_{  \ell  } \bfR^+,
h_1 \bfX_1, \ldots, h_k \bfX_k, h_{k+1} \bfX_{k+1},\ldots,  h_r \bfX_r) \\
  \underset{(**)}{\longleftrightarrow}\ &  
(j_1 \bfR^+,\ldots, j_{ \ell } \bfR^+,
h_1 \bfX_1, \ldots, h_{k+1} \bfX_{k+1}, h_k \bfX_k,\ldots,  h_r \bfX_r)
= \bfs'.
\end{align*}

By Remark \ref{Rmk: cases}, we have 
\begin{itemize}
\item $(**)$ is (Case 5) if  $h_{k+1}=h_k^*$, i.e., $\bcol(\frakc_k)=\bcol(\frakc_{k+1})$,
\item $(**)$ is (Case 4) if  $h_{k+1} \ne h_k^*$, i.e., $\bcol(\frakc_k)\ne\bcol(\frakc_{k+1})$,
\end{itemize}
which implies the assertion by Theorem~\ref{thm:cases}.
\end{proof}

The following lemma tells us how cluster variables are changed under the weave $\fW_{\rxw}^{\rm T}(\frakC) : \rxw\uii \rightarrow \uii_\rxw (\frakC)$ given in \eqref{Eq: fWT}.

\begin{lemma}\label{lem:cluster_variables_modified_weave}
Let $\frakC = (\frakc_t)_{t\in [1,r]}$ be an admissible chain associated with $\uii \in \Seq(\ttb)$. Consider the two braid varieties 
$$
X(\uii_\rxw (\frakC)) \qtq X(\rxw\uii),
$$ 
and let $y := (y_1,\ldots,y_{\ell+r})$ and $z := (z_1, \ldots, z_{\ell+r})$ denote the coordinates associated to $X(\uii_\rxw (\frakC))$ and $X(\rxw\uii)$, respectively (see \eqref{Eq: Br coord}).
We set 
$$
\fW^T := \fW_\rxw^T(\frakC), \quad  \fW^B := \fW_\rxw^B(\frakC) \qtq \fW := \fW_\rxw(\frakC),
$$
and consider the cluster variables arising from $\fW^B$ and $\fW$ as follows:
\begin{align*}
\calA \left(\fW^B\right) &:= \left\{ A_t \left(\fW^B \right) \right\}_{ t\in [1,r]} \subset \C[X(\uii_\rxw (\frakC))], \\
\calA(\fW ) &:= \left\{ A_t (\fW ) \right\}_{ t\in [1,r]} \subset \C[X(\rxw\uii)].
\end{align*}
\bnum
\item 
There exist polynomials $f_k (z) \in \mathbb{Z}[z_1, \ldots, z_{\ell+r}]$ for $k \in [1,\ell+r]$ such that
\[
A_t (\fW)(z) = A_t \left(\fW^B \right) \left(f_1(z),f_2(z), \ldots, f_{\ell+r}(z) \right) \quad \text{for all } t \in [1,r].
\]
Moreover, the polynomials $f_k(z)$ depend only on $\fW^T$.
\item \label{lem:cluster_variables_modified_weave (ii)}
We rewrite the coordinates $z = (w,z) = (w_1, \ldots, w_{\ell}, z_1, z_2, \ldots, z_{r})$ for $X(\rxw\uii)$.
If $\frakc_1 = [s,s]$,  then we have 
$$
A_1 (\fW) = z_{s}.
$$ 
\ee
\end{lemma}
\begin{proof}
 By construction, $\fW =  \fW^B \circ \fW^T $ and $\fW^T$ consists of $6$-valent and $4$-valent vertices, and contains no $3$-valent vertices.

(i) We trace the coordinates $z = (z_1, \ldots, z_{\ell+r})$ downward through $\fW^T$ using the assignment procedure of \cite[Section~5.2]{CGGLSS25}. By \cite[Theorem~5.12]{CGGLSS25}, there exists a labeling $\zeta$ of $\fW$ producing cluster variables for $\fW$ (see \cite[Definition 5.3]{CGGLSS25} for labeled weaves). 
Let $E$ be the set of all (solid) edges of $\fW$ and write $\zeta(e) = (\tilde{z}_e, u_e)$ for $e \in E$, where $\tilde{z}_e$ and $ u_e $ are rational functions defined by using the inductive relations in Definition \cite[Definition 5.8]{CGGLSS25}. 
If $e_i \in E$ is the edge corresponding to the $i$th strand on the top of $\fW$, then we have $\zeta(e_k) =(\tilde{z}_{e_k}, u_{e_k}) = (z_k, 1)$, which serves as the initial data to determine other values $\zeta(e)$ in $\fW$.   

The functions $u$ depend entirely on the Lusztig cycles originating from $3$-valent vertices. 
Since there are no $3$-valent vertices in $\fW^T$, the functions $u$ are identically $1$ at all edges in $\fW^T$. 
Under the condition $u=1$, the inductive relations for the functions $\tilde{z}$ at $6$-valent and $4$-valent vertices (\cite[Definition~5.8]{CGGLSS25}) degenerate to polynomial relations over $\mathbb{Z}$. 
The polynomial relations are described  in Figure~\ref{fig:z-variable-change}.
Consequently, propagating the coordinates $z$ to the bottom boundary of $\fW^T$ yields variables of the form 
\begin{align} \label{Eq: (z,u)}
(\tilde{z}_{e_k}, u_{e_k}) = (f_k(z), 1) \qquad \text{ for some $f_k \in \mathbb{Z}[z_1, \ldots, z_{\ell+r}]$.}
\end{align}
Since this bottom boundary coincides with the top of the weave $\fW^B$, the uniqueness of the downward inductive procedure ensures that the cluster variables $A_t (\fW)(z)$ coincide with those obtained from $A_t \left(\fW^B \right)$ by substituting $y$ with $(f_1(z), \ldots, f_{\ell+r}(z))$.

(ii) By definition, the length of $\uii^L(\frakC)$ is $s-1$ and 
$$
\uii_\rxw(\frakC) = \uii^L(\frakC) \  \rxw \  \uii^R(\frakC).
$$
Since $\sigma_{i^*} \Delta =  \Delta \sigma_{i}$ for any $i\in I$, we can choose $\fW^T: \rxw \uii \to \uii_\rxw(\frakC)$ such that there are no vertices involving the last $r-s+1$ strands in $\fW^T$.  
From \eqref{Eq: (z,u)}, together with the coordinates \( (w_1,\ldots,w_\ell,z_1,\ldots,z_r) \) for \(X(\rxw\uii)\), 
we have
$$
f_{\ell+s} = z_s
$$ 
i.e., $(\tilde{z}_{e_{\ell + s}}, u_{e_{\ell + s}}) = (z_s, 1)$.
On the other hand, by \cite[Definition 5.8, Theorem 5.19 (3) and Remark 6.7]{CGGLSS25}, we have 
$$
A_1 (\fW^B) = y_{\ell+s}, 
$$
which yields the assertion by (i).
\end{proof}

\begin{figure}[htbp]
  \centering
  \begin{tikzpicture}[scale=1]
    \draw [ultra thick] 
         (0,0) -- (0,-1) (0,0) -- (-1,1) (0,0) -- (1,1);
         \draw [ultra thick, color=purple] 
         (0,0) -- (0,1) (0,0) -- (-1,-1) (0,0) -- (1,-1);
         \node[above] at (-1.3,1) {$\tilde{z}$};
         \node[above] at (0,1) {$\tilde{z}'$};
         \node[above] at (1.3,1) {$\tilde{z}''$};

         \node[below] at (-1.3,-1) {$\tilde{z}''$};
         \node[below] at (0,-1) {$\tilde{z}\tilde{z}'' - \tilde{z}'$};
         \node[below] at (1.3,-1) {$\tilde{z}$};
  \end{tikzpicture}
   \qquad \qquad \qquad
  \begin{tikzpicture}[scale=1]
    \draw [ultra thick] 
         (-1,1) -- (1,-1);
         \draw [ultra thick, color=teal] 
         (1,1) -- (-1,-1);
         \node[above] at (-1.3,1) {$\tilde{z}$};
         \node[above] at (1.3,1) {$\tilde{z}'$};

         \node[below] at (-1.3,-1) {$\tilde{z}'$};
         \node[below] at (1.3,-1) {$\tilde{z}$};
  \end{tikzpicture}  
  \caption{Transformations of the $\tilde{z}$-variables at $6$-valent and $4$-valent vertices under the condition $u=1$.}
  \label{fig:z-variable-change}
\end{figure}
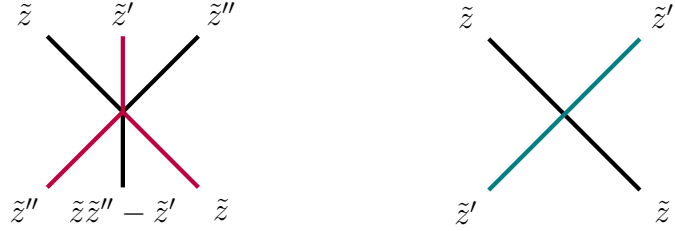

We now have the main theorem of this subsection. 
\begin{theorem}\label{thm: modified weave boxmove}
Let $\frakC$ be an admissible chain of $i$-boxes associated with $\uii \in \Seq(\ttb)$ and let $\frakc_k$ be a movable box. For any $\rxw \in \Seq(\Delta)$, we have 
\begin{equation*}
    Q\bigl(\fW_\rxw(\calB_k(\frakC))\bigr) =
        \begin{cases}
            \mu_k\bigl(Q( \fW_\rxw(\frakC))\bigr)
            & \text{if $\bcol(\frakc_k)=\bcol(\frakc_{k+1})$}, \\[0.3em]
            \sigma_{k,k+1}\bigl(Q(\fW_\rxw(\frakC))\bigr)
            & \text{if $\bcol(\frakc_k)\neq\bcol(\frakc_{k+1})$},
        \end{cases}
\end{equation*}
as labeled quivers and
\begin{equation*}
\calA\bigl(\fW_\rxw(\calB_k(\frakC))\bigr)=
\begin{cases}
\mu_k\bigl(\calA( \fW_\rxw(\frakC))\bigr)
& \text{if $\bcol(\frakc_k)=\bcol(\frakc_{k+1})$},\\[0.3em]
\sigma_{k,k+1}\bigl(\calA( \fW_\rxw(\frakC))\bigr)
& \text{if $\bcol(\frakc_k)\neq\bcol(\frakc_{k+1})$}.
\end{cases}
\end{equation*}
\end{theorem}

\begin{proof}
Since the weave $\fW^T_\rxw( \frakC)$ has no 3-valent vertices, it does not contribute any vertices to the associated quiver. Hence
$$
Q(\fW_\rxw( \frakC)) = Q(\fW^B_\rxw( \frakC))
$$
as labeled quivers. Thus the equality for quivers follows directly from Proposition \ref{prop: double string boxmove}.

We now focus on the equality for cluster variables. 
Since the weave equivalence class of $ \fW_{\rxw}(\frakC)$ is independent of the choice $\rxw \in \Seq(\Delta)$, there is no harm in verifying the assertion using specific expressions by Proposition \ref{Prop: we seed}.

Let $\rxw$ and $\rxw'$ be the expression sequences chosen as in Proposition~\ref{prop: double string boxmove} if $k=1$ and $\rxw = \rxw'$ otherwise. 
With this choice, the $\uii$-sequences on the tops of the weaves $\fW^B_{\rxw} ( \frakC )$ and $\fW^B_{\rxw'} ( \calB_k(\frakC))$ coincide. 
Thus $\fW(\frakC)$ and $ \fW (\calB_k(\frakC))$ can be constructed by attaching the same weave $\fW^T_{\rxw}(\frakC)$ to the tops of $\fW^B_{\rxw} ( \frakC )$ and $\fW^B_{\rxw'} ( \calB_k(\frakC))$, respectively.
Lemma~\ref{lem:cluster_variables_modified_weave} says that the weave $\fW^T_{\rxw}(\frakC)$ induces the same polynomial substitution on the coordinates. 
Since the substitution of cluster variables by the same polynomials is compatible with quiver mutations and permutations, Proposition~\ref{prop: double string boxmove} induces the assertion. 
\end{proof}

\subsection{Cluster algebra isomorphisms}

Recall the bosonic extension $\tobbA_\C(\ttb)$ and its seed $\seed(\frakC)$ given in Section \ref{Sec: BE}.

\begin{theorem}\label{thm: main sec3}
Let $\g$ be a simple Lie algebra of finite $ADE$ type and let $\ttb \in \Br^+$.   
We take $\rxw \in \Seq(\Delta)$, $\uii \in \Seq(\ttb)$ and set $r:= \ell(\uii)$. Then there exists a unique algebra isomorphism
\[
\varphi_{\uii} \colon \tobbA_\C(\ttb) \xrightarrow{\ \sim\ }
\C[X(\rxw \uii)]
\]
such that 
\bna
\item for any $k \in [1,r]$, 
$$
\varphi_{\uii} (\bFF_{\uii, k})  = z_{k}, 
$$ 
where $\bFF_{\uii, k}$ are the PBW vectors in $\tobbA_\C(\ttb)$ defined in \eqref{Eq: PBW vector} and 
$z = (z_1, \ldots, z_r)$ are the coordinates of $\C[X(\rxw \uii)]$ via the isomorphism \eqref{Eq: 3 iso}, 
\item for any admissible chain $\frakC$ of $i$-boxes associated with $\uii$, the two seeds 
\[
\seed(\frakC) = \bigl( \mD(\frakC), Q(\frakC)\bigr)  \qtq
\Sigma\bigl( \fW_\rxw (\frakC)\bigr) = \bigl( \calA(\fW_\rxw(\frakC)), Q( \fW_\rxw (\frakC)) \bigr)
\]
are identified by $\varphi_{\uii}$ in the following sense:
$$
\calA\bigl(\fW_\rxw(\frakC)\bigr)
=\varphi_{\uii}\bigl(\mD(\frakC)\bigr) \qtq
Q\bigl( \fW_\rxw(\frakC)\bigr) = Q(\frakC)^{\rm op}
$$
as labeled quivers.
\ee
\end{theorem}

\begin{proof}
For an admissible chain  $\frakC = (\frakc_t)_{t\in [1,r]}$, we write  $\mD_t(\frakC) := \mD(\frakc_t)$ for $t\in [1,r]$ and 
$$ \mD(\frakC) = \{ \mD_t(\frakC) \}_{t\in [1,r]} \qtq \calA(\fW_\rxw(\frakC)) = \{ A_t(\fW_\rxw(\frakC)) \}_{t\in [1,r]}.
$$

Since the weave $\fW^T_\rxw( \frakC)$ has no 3-valent vertices, we have 
$
Q(\fW_\rxw( \frakC)) = Q(\fW^B_\rxw( \frakC))
$
as labeled quivers. Lemma \ref{lem: RRR and riw} says that 
\begin{align} \label{Eq: QCQX}
Q(\frakC_{\calR}) = Q(\fW_\rxw(\frakC_\calR))^{\rm op}.
\end{align}
Since both of $\tobbA_\C(\ttb)$ and $\C[X(\rxw \uii)]$ have cluster algebra structures with the initial quivers $Q(\frakC_{\calR})$ and $Q(\fW_\rxw(\frakC_\calR))$ respectively, 
and they are isomorphic to the localizations of the polynomial rings (with the same number of variables) by frozen variables (see \eqref{Eq: iso PA} and \eqref{Eq: 3 iso}), by \eqref{Eq: QCQX}, we have the algebra isomorphism
\[
\varphi_{\uii} \colon \tobbA_\C(\ttb) \xrightarrow{\ \sim\ } \C[X(\rxw \uii)]
\]
defined by  identifying the initial cluster variables, i.e., 
\[
\varphi_{\uii}\bigl(\mD_t(\frakC_{\calR})\bigr) = A_t(\fW_\rxw(\frakC_\calR)) \quad \text{for all } t \in [1,r].
\]
By construction, it is obvious that $\varphi_{\uii}$ is unique. 
Since any admissible chain $\frakC$ can be obtained from the chain $\frakC_{\calR}$ via a finite sequence of box moves,
Propositions~\ref{prop:boxmove_seed} and Theorem \ref{thm: modified weave boxmove} imply (b).

Let  $s\in [1,r]$, and choose an admissible chain $\frakC$ such that $\frakc_1=[s,s]$.
Since $\mD_1(\frakC)=\bFF_{\uii,s}$
by the definition \eqref{Def: minors},
Lemma~\ref{lem:cluster_variables_modified_weave}
\eqref{lem:cluster_variables_modified_weave (ii)}
gives (a).
\end{proof}

\begin{example}\label{ex:weave_tilde}
Let $\g$ be the simple Lie algebra of type $A_3$, and consider the following
$$
\rxw=(1,2,1,3,2,1) \qtq \uii = (2,3,1,2,2,1).
$$
\bnum
\item We consider the admissible chain $\frakC_\calR = (\frakc_t')_{t\in [1,6]}$ determined by the LR sequence $\frakH' = (\calR,\calR,\calR,\calR,\calR)$ as in Example \ref{ex: ibox exchange matrix} \eqref{ex: ibox exchange matrix (i)}. Then the quiver $Q(\frakC_\calR)$ and cluster variables $\mD(\frakC_\calR)$
are given in Example \ref{ex: ibox exchange matrix} \eqref{ex: ibox exchange matrix (i)} and Example \ref{ex: ibox seed} \eqref{ex: ibox seed (i)} respectively. Note that $\fW_\rxw(\frakC_\calR)  = \rind{\rxw\uii}$.
Theorem \ref{thm: main sec3} says that 
$$
\calA\bigl(\fW_\rxw(\frakC_\calR)\bigr)
=\varphi_{\uii}\bigl(\mD(\frakC_\calR)\bigr) \qtq
Q\bigl( \fW_\rxw(\frakC_\calR)\bigr) = Q(\frakC_\calR)^{\rm op}.
$$
Note that $\calA\bigl(\fW_\rxw(\frakC_\calR)\bigr)$ is computed explicitly in Example \ref{ex:initial_variables} and $Q\bigl( \fW_\rxw(\frakC_\calR)\bigr)$ is given in Example \ref{ex:initial_variables 1}.

\item 
We continue with Example~\ref{ex:double_weave}.
Let $\frakC = (\frakc_t)_{t\in [1,6]}$  be  the admissible chain determined by the LR sequence $\frakH = (\calL,\calR,\calL,\calR,\calR)$. Then the quiver $Q(\frakC)$ and cluster variables $\mD(\frakC)$
are given in Example \ref{ex: ibox exchange matrix} \eqref{ex: ibox exchange matrix (ii)} and Example \ref{ex: ibox seed} \eqref{ex: ibox seed (ii)} respectively.
We take the weave $\fW^{T}_\rxw(\frakC) : \rxw\,\uii \to \uii_\rxw(\frakC)$ in Figure \ref{fig:weave1} so that  the weave $\fW_\rxw(\frakC)$ can be obtained by concatenating the weaves $\fW^{T}_\rxw(\frakC)$ and $\fW^{B}_\rxw(\frakC)$ given in Figure~\ref{fig:weave_quiver}. Note that the quiver $Q(\fW_\rxw(\frakC)) = Q(\fW_\rxw^B(\frakC))$ is given in Figure~\ref{fig:weave_quiver}.
By Theorem \ref{thm: main sec3}, we have 
$$
\calA\bigl(\fW_\rxw(\frakC)\bigr) 
=\varphi_{\uii}\bigl(\mD(\frakC)\bigr) = (A_t)_{t\in [1,6]} \qtq
Q\bigl( \fW_\rxw(\frakC)\bigr) = Q(\frakC)^{\rm op},
$$
where the cluster variables $A_t$ are given as follows:
\bna
\item $A_1 = \varphi_{\uii} (\mD(\frakc_1)) = z_3$,
\item $A_2 = \varphi_{\uii} (\mD(\frakc_2)) = z_2$,
\item $A_3 = \varphi_{\uii} (\mD(\frakc_3)) = z_4$,
\item $A_4 = \varphi_{\uii} (\mD(\frakc_4)) = z_1z_4 - z_2z_3$,
\item $A_5 = \varphi_{\uii} (\mD(\frakc_5)) = z_1z_4z_5 - z_2z_3z_5 - z_1$,
\item $A_6 = \varphi_{\uii} (\mD(\frakc_6)) = z_3z_6 - z_4z_5 + 1$.
\ee

\ee
\end{example}

\begin{figure}[htbp] 
  \begin{tikzpicture}[scale=0.8]
	\draw [ultra thick] 
         (0/2,10) -- (0/2,9) -- (1/2,8.5) -- (1/2,8)
         (2/2,10) -- (2/2,9) -- (1/2,8.5)
         (5/2,10) -- (5/2,8)
         (8/2,10) -- (8/2,8)
         (11/2,10) -- (11/2,8)
         (1/2,8) -- (1/2,3)
         (8/2,8) -- (8/2,3)
         (11/2,8) -- (11/2,3)
         (5/2,8) -- (5/2,7.5) -- (4/2,7) -- (4/2,6) -- (3/2,5) -- (3/2,3.5) -- (2/2,3)
         (3/2,3.5) -- (4/2,3)
         (5/2,7.5) -- (6/2,7) -- (7/2,6) -- (7/2,3);
         \draw [ultra thick, color=purple] 
	(1/2,10) -- (1/2,8.5) -- (0/2,8)
	(1/2,8.5) -- (2/2,8)
	(4/2,10) -- (4/2,8)
	(6/2,10) -- (6/2,8)
	(9/2,10) -- (9/2,8)
	(10/2,10) -- (10/2,8)
	(0/2,8) -- (0/2,3)
	(9/2,8) -- (9/2,3)
	(10/2,8) -- (10/2,3)
	(2/2,8) -- (2/2,4) -- (3/2,3.5) -- (3/2,3)
	(4/2,8) -- (5/2,7.5) -- (5/2,4.5) -- (4/2,4) -- (3/2,3.5)
	(6/2,8) -- (5/2,7.5)
	(5/2,4.5) -- (6/2,4) -- (6/2,3);
	\draw [ultra thick, color=teal]
	(3/2,10) -- (3/2,9) -- (3/2,8)
	(7/2,10) -- (7/2,8)
	(3/2,8) -- (3/2,6) -- (4/2,5) -- (5/2,4.5) -- (5/2,3)
	(7/2,8) -- (7/2,7) -- (6/2,6) -- (6/2,5) -- (5/2,4.5) -- (5/2,3);
         \foreach \i/\j in {0/1, 1/2, 2/1, 3/3, 4/2, 5/1, 6/2, 7/3,8/1,9/2,10/2,11/1}{
         \node[above] at (\i/2,10) {$\j$};
         }
          \foreach \i/\j in  {0/2, 1/1, 2/1, 3/2, 4/1, 5/3, 6/2, 7/1,8/1,9/2,10/2,11/1}{
         \node[above] at (\i/2,2.3) {$\j$};
         }
  \end{tikzpicture}
    \caption{A weave $\fW^{T}_\rxw (\frakC)$  in Example~\ref{ex:weave_tilde}.}
    \label{fig:weave1}
\end{figure}

\newpage

\section{Connection to signed words}\label{sec: connection to signed words}

In this section, we review the connection between $i$-boxes and signed words studied in \cite{CQW26}, and obtain an explicit connection among $i$-boxes, weaves and signed words through the isomorphism given in Theorem~\ref{thm: main sec3}.

Let $\g$ be a simple Lie algebra associated with a Cartan matrix $\cmC = (c_{i,j})_{i,j\in I}$ of \emph{finite $ADE$} type. A \emph{signed word} is a sequence of pairs
\[
\uh = (\epsilon_1 h_1, \epsilon_2 h_2, \ldots, \epsilon_r h_r),
\]
where $h_t \in I$ and $\epsilon_t \in \{\pm 1\}$ for $1 \le t \le r$.  
Given a signed word $\uh=(\epsilon_t h_t)_{1\leq t \leq r}$, we define the index sets
\begin{align*}
\sfK(\uh) &:= \{1,2,\ldots,r \},\\
\sfK(\uh)^{\rm fr} &:= \{ j\in \sfK(\uh) \mid j_\bfh^+ = +\infty \}, \\
\sfK(\uh)^{\rm ex} &:= \sfK(\uh)\setminus \sfK (\uh)^{\rm fr},
\end{align*}
where $\bfh = ( h_1, h_2, \ldots, h_r )$.
For a signed word $\uh$, the corresponding exchange matrix $B(\uh) = (b_{j,t})_{j\in \sfK(\uh),\ t \in  \sfK(\uh)^{\rm ex}}$ is defined by 
\begin{equation*}
b_{j,t} = \begin{cases}
\epsilon_{t} & \text{if } t=j^+, \\
-\epsilon_{j} & \text{if } j=t^+, \\
\epsilon_{t}c_{h_j,h_t} & \text{if } \epsilon_{j^+}=\epsilon_{t} \text{ and } j<t<j^+<t^+, \\
\epsilon_{t}c_{h_j,h_t} & \text{if } \epsilon_{t}=-\epsilon_{t^+} \text{ and } j<t<t^+<j^+, \\
-\epsilon_{j}c_{h_j,h_t} & \text{if } \epsilon_{t^+}=\epsilon_{j} \text{ and } t<j<t^+<j^+, \\
-\epsilon_{j}c_{h_j,h_t} & \text{if } \epsilon_{j}=-\epsilon_{j^+} \text{ and } t<j<j^+<t^+, \\
0 & \text{otherwise}.
\end{cases}
\end{equation*}
We denote by $Q(\uh)$ the quiver associated with the exchange matrix $B(\uh)$.

Let $\ttb \in \Br^+$ and let  $\uii \in \Seq(\ttb)$. We consider an admissible chain $\frakC=(\frakc_t)_{1\le t \le r}$ of $i$-boxes associated with $\uii$, and denote by $\frakH=(\calH_1,\ldots,\calH_{r-1})$ the corresponding LR sequence.
The \emph{associated signed word} 
$$
\uh(\frakC):=(\epsilon_t h_t)_{1\leq t \leq r}
$$
is defined by
\[
h_t := \bcol(\frakc_t) \qtq
\epsilon_t :=
\begin{cases}
+1 & \text{if } t=1 \text{ or } \calH_{t-1}=\calL, \\
-1 & \text{if } \calH_{t-1}=\calR.
\end{cases}
\]
It was shown in \cite[Theorem 1.1]{CQW26} that the quivers 
$Q(\frakC)$ and $Q\bigl(\uh(\frakC)\bigr)$ coincide for an admissible chain $\frakC$ of $i$-boxes associated with $\uii \in \Seq(\ttb)$. 
Combining this with Theorem \ref{thm: main sec3}, we have an explicit connection among $i$-boxes, weaves and signed words. 
\begin{corollary}\label{cor:signed-word-weave-quiver}
Let \( \rxw \in \Seq(\Delta) \), and let  $\frakC$ be an admissible chain of $i$-boxes associated with $\uii \in \Seq(\ttb)$. Then 
we have 
\[
Q(\frakC) = Q\bigl(\uh(\frakC)\bigr)=Q\bigl(\fW_{\rxw}(\frakC)\bigr)^{\rm op}
\]
as labeled quivers.
\end{corollary}

\vskip 2em


\section{Connection to quantum affine algebras}\label{sec: connection to qaa}

Let $U_q'(\g)$ be a \emph{quantum affine algebra} over an algebraic closure of $\C(q)$, where $q$ is an indeterminate. We denote by $\Cg$ the category of finite-dimensional integrable $U_q'(\g)$-modules and by $\Cgz$ the \emph{Hernandez-Leclerc} category, which is a distinguished subcategory of $\Cg$ (see \cite{HL10, HL15}, \cite[Section 2]{KKOP24A} and \cite[Section 2.2]{KKOP25} for details).
To each $U_q'(\g)$, we associate a simply-laced finite simple Lie algebra $\gf$ as listed in Table~\ref{tab: g_sfg_corr} (see \cite{KKOP22} and see also \cite[Section 2]{KKOP25}). We use the notations $I_\fin$, $\cmC_\fin$, etc., for the Lie algebra $\gf$.

\begin{table}[h]
\centering
\begin{tabular}{c|c}
\hline
affine type  $\g$ & simply-laced  finite type $\gf$ \\ \hline
$A_1^{(1)}$ & $A_1$ \\ 
$A_{2n}^{(1)},\ A_{2n}^{(2)}$ & $A_{2n}\ (n\ge 1)$ \\ 
$A_{2n-1}^{(1)},\ A_{2n-1}^{(2)},\ B_n^{(1)}$ & $A_{2n-1}\ (n\ge 2)$ \\ 
$D_4^{(1)},\ D_4^{(2)},\ C_3^{(1)},\ D_4^{(3)},\ G_2^{(1)}$ & $D_4$ \\ 
$D_{n+1}^{(1)},\ D_{n+1}^{(2)},\ C_n^{(1)}$ & $D_{n+1}\ (n\ge 4)$ \\ 
$E_6^{(1)},\ E_6^{(2)},\ F_4^{(1)}$ & $E_6$ \\ 
$E_7^{(1)}$ & $E_7$ \\ 
$E_8^{(1)}$ & $E_8$ \\ \hline
\end{tabular}
\caption{Affine type $\g$ and simply-laced finite type $\gf$}
\label{tab: g_sfg_corr}
\end{table}

We take a \emph{complete duality datum} $\bbD=\{L_i^\bbD\}_{i \in I_\fin} \subset \Cgz$, which is the set of \emph{root modules} satisfying certain categorical conditions determined by $\cmC_\fin$ (see \cite[Definition 2.10]{KKOP25}).  Let $\hcalA$ be the bosonic extension associated with $\cmC_\fin$. It was shown in \cite[Proposition 3.18]{KKOP25} that there exists a homomorphism 
\[
\bPhi_\bbD \colon \hcalA_\Zq \to K(\Cgz)    
\]
satisfying $\bPhi_\bbD(f_{i,p}) = \bigl[\rdual^p(L_i^\bbD)\bigr]$ for any $(i,p) \in I_\fin \times \Z$ and $\bPhi_\bbD(q^{1/2}) = 1$, where $\rdual$ is the right dual functor in $\Cgz$.
Specializing $\obbA := \hcalA_\Zq / (q^{1/2}-1) \hcalA_\Zq$,
\footnote{$\obbA$ and $\obbA(\ttb)$ are denoted by  $^{\circ}\obbA$ and $^{\circ}\obbA(\ttb)$ respectively in \cite{KKOP25}.}
we obtain the induced isomorphism 
\[
{^\circ}\bPhi_\bbD \colon \obbA \to K(\Cgz)    
\]
(see \cite[Theorem 3.19]{KKOP25}).

We denote by $\Br$ the Braid group associated with $\cmC_\fin$. Let $\ttb \in \Br^+$ and let 
$$
\uii = (i_1, i_2, \ldots, i_r)  \in \Seq(\ttb).
$$ 
For each $t\in [1,r]$, we consider the \emph{affine cuspidal module} $C^{\bbD,\uii}_t$ defined in \cite[(4.4)]{KKOP25}.  For an $i$-box $[a,b]$ of $\uii$, we define the \emph{affine determinantial module} $M^{\bbD,\uii}[a,b]$ by
    \[
    M^{\bbD,\uii}[a,b] := \head\left(C^{\bbD,\uii}_{b} \tens C^{\bbD,\uii}_{b^-} \tens \cdots \tens C^{\bbD,\uii}_{a^+} \tens C^{\bbD,\uii}_a\right)
    \]
(see \cite[Definition 4.8]{KKOP25}). The affine determinantial modules are real simple modules in $\Cgz$ and satisfy the $T$-system relations.
We simply write $C^{\uii}_{t}$ and $ M^{\uii}[a,b]$ for $C^{\bbD,\uii}_{t}$ and $ M^{\bbD,\uii}[a,b]$ respectively if no confusion arises. 

\begin{theorem}[{\cite[Theorem 5.15]{KKOP25}}] \label{thm:Tsystem}
For any $i$-box $[a,b]$ of $\uii$, there exists the following short exact sequence
$$ 
0 \to  \bigotimes_{ \substack{ j \in I_\fin \\ (\al_{i_a},\al_j)=-1}}
M^{\uii}[a(j)^+,b(j)^-]
\to
M^{\uii}[a^+,b] \tens M^{\uii}[a,b^-]
\to
M^{\uii}[a,b]\tens M^{\uii}[a^+,b^-]
\to 0.
$$
\end{theorem}

 We define $\scrC_\g^{\bbD} (\ttb)$ to be the smallest full subcategory of $\Cgz$ such that 
 \bna
 \item $\scrC_\g^{\bbD} (\ttb)$ contains all affine cuspidal modules $\{ C_t^{\bbD,\uii} \}_{1 \le t \le r}$ and the trivial module $\bone$,
 \item $\scrC_\g^{\bbD} (\ttb)$ is stable under taking tensor products, subquotients, and extensions.
\ee 
The category $\scrC_\g^{\bbD} (\ttb)$ provides a monoidal categorification of $\obbA(\ttb)$ using the affine determinantial modules (\cite{KKOP25}).

\begin{proposition}[{\cite[Corollary 5.35, Theorem 9.4]{KKOP25}}] \label{cor: K simeq oA}
Let $\ttb \in \Br^+$ and let $\uii \in \Seq(\ttb)$.
\bnum
\item There is a unique algebra isomorphism 
$$
K(\scrC_\g^{\bbD}(\ttb)) \simeq \obbA(\ttb)
$$
such that
\begin{align*}
[C^{ \uii}_t] &\mapsto \bFF_{\uii, t} \qquad\qquad \text{ for any $t\in [1,r]$,} \\
[M^{ \uii}[a,b]]  &\mapsto \mD_{\uii}[a,b] \ \qquad \text{ for any $i$-box $[a,b]$ of $\uii$,}
\end{align*}
where $\bFF_{\uii, t}$ is the PBW vector defined in \eqref{Eq: PBW vector}, and $\mD_{\uii}[a,b]$ is given in \eqref{Def: minors}.
\item For each admissible chain $\frakC = (\frakc_t)_{t\in [1,r]}$ of  $i$-boxes associated with $\uii$, the pair 
$$\seed(\frakC) = ( \{ M^{ \uii}(\frakc_t) \}_{t\in [1,r]} , B(\frakC))
$$
forms an admissible monoidal seed for $\scrC_\g^{\bbD}(\ttb)$, which gives a categorical cluster algebra structure of $\scrC_\g^{\bbD}(\ttb)$. Therefore, $\scrC_\g^{\bbD} (\ttb)$ provides a monoidal categorification of $\obbA(\ttb)$.
\ee
\end{proposition}

Combining this with Theorem~\ref{thm: main sec3}, we realize the cluster structure on $\C[X(\Delta \uii)]$ at the level of the monoidal category $\scrC_\g^{\bbD}(\ttb)$.

\begin{corollary}
Let $\ttb \in \Br^+$ and let $\uii \in \Seq(\ttb)$. Then the category $\scrC_\g^{\bbD}(\ttb)$ provides a monoidal categorification of the partially compactified cluster algebra of the coordinate ring $\C[X(\Delta \uii)]$ (see Definition \ref{Def: cluster alg} \eqref{Def: cluster alg (ii)}).
Moreover, the affine cuspidal modules $C^{ \uii}_t$ categorify the coordinates $z_t \in \C[X(\Delta \uii)]$ under this categorification. 
\end{corollary}

The isomorphism $\varphi_{\uii}$ given in Theorem~\ref{thm: main sec3} together with Proposition \ref{prop: conf variables} yields an explicit formula for the decomposition of $[M^{ \uii}[a,b]]$ in terms of $[C_{t}^{ \uii}]$ in the Grothendieck ring $K(\Cg^{\bbD}(\ttb))$.

\begin{corollary} \label{Cor: formular for M and C}
For any $\uii \in \Seq(\ttb)$ and an $i$-box $[a,b]$ of $\uii$, we have  
\[
[M^{ \uii}[a,b]]  = \triangle_{i_a}\bigl(B_{i_a}([C_{a}^{ \uii}])\, B_{i_{a+1}}( [C_{a+1}^{ \uii}]) \cdots B_{i_b}( [C_{b}^{ \uii}])\bigr),
\]
where $\triangle_{i}$ denotes the generalized principal minor associated to the fundamental weight $\varpi_{i}$, and $B_i(z)$ is defined as in \eqref{eq: B_i(z)}.
\end{corollary}
\begin{proof}
By the \(T\)-system in Theorem~\ref{thm:Tsystem}, together with the initial
conditions \([M^{\uii}[k,k]]=[C_k^{\uii}]\), the class
\([M^{\uii}[a,b]]\) is determined by the subword \((i_a,\ldots,i_b)\) and the
classes \([C_a^{\uii}],\ldots,[C_b^{\uii}]\) through the same recursion as in
the case of an interval starting at \(1\).  Thus, after relabeling the interval
\([a,b]\), it suffices to prove the case \(a=1\), which follows from
Proposition~\ref{prop: conf variables}, Theorem~\ref{thm: main sec3} and Proposition~\ref{cor: K simeq oA}.
\end{proof}

\begin{example}
Assume that $\g$ is of affine type $A_3^{(1)}$. 
For the notation used in this example, we refer the reader to \cite[Section 8]{KKOP25A} and \cite[Section 5]{HL10}.
Then $\gf$ is of finite type
$A_3$, and the corresponding Lie group is $\SL_4(\C)$. We set
\[
  \uii=(2,3,1,2,2,1).
\]
Let \(L(m)\) denote the simple module with dominant monomial
\(m\), written in the variables \(Y_{i,p}= Y_{i, (-q)^p}\).
Take the
complete duality datum
\[
  \bbD=\{L^{\bbD}_1=L(Y_{3,6}),\quad
  L^{\bbD}_2=L(Y_{2,3}),\quad
  L^{\bbD}_3=L(Y_{1,6})\}
\]
in \(\Cgz\).  For this choice, the affine cuspidal modules attached to \(\uii\)
are
\[
\begin{array}{lll}
  C_1^{\uii}=L(Y_{2,3}),&
  C_2^{\uii}=L(Y_{3,4}),&
  C_3^{\uii}=L(Y_{1,4}),\\[2mm]
  C_4^{\uii}=L(Y_{2,5}),&
  C_5^{\uii}=L(Y_{2,9}),&
  C_6^{\uii}=L(Y_{1,6}Y_{2,9}).
\end{array}
\]

Consider the \(i\)-box \(\frakc=[3,6]\), and set 
$C_t:=[C_t^{\uii}]$ for $1\le t\le 6$,  
\[
  M:=[M^{\uii}[3,6]]= [ \hd(C_6^{\uii} \tens C_3^{\uii}) ] = [L(Y_{1,4} Y_{1,6} Y_{2,9} )].
\]
The formula in Corollary~\ref{Cor: formular for M and C} gives
\[
  M=
  \triangle_{1}\Bigl(
    B_{1}(C_3)\,B_{2}(C_4)\,B_{2}(C_5)\,B_{1}(C_6)
  \Bigr).
\]
Here
\[
B_{1}(C_3)=
\begin{pmatrix}
C_3&-1&0&0\\
1&0&0&0\\
0&0&1&0\\
0&0&0&1
\end{pmatrix},\quad
B_{2}(C_4)=
\begin{pmatrix}
1&0&0&0\\
0&C_4&-1&0\\
0&1&0&0\\
0&0&0&1
\end{pmatrix},
\]
and
\[
B_{2}(C_5)=
\begin{pmatrix}
1&0&0&0\\
0&C_5&-1&0\\
0&1&0&0\\
0&0&0&1
\end{pmatrix},\quad
B_{1}(C_6)=
\begin{pmatrix}
C_6&-1&0&0\\
1&0&0&0\\
0&0&1&0\\
0&0&0&1
\end{pmatrix}.
\]
A direct multiplication yields
\[
B_{1}(C_3)\,B_{2}(C_4)\,B_{2}(C_5)\,B_{1}(C_6)
=
\begin{pmatrix}
C_3C_6-C_4C_5+1 & -C_3 & C_4&0\\
C_6&-1&0&0\\
C_5&0&-1&0\\
0&0&0&1
\end{pmatrix}.
\]
Since \(\triangle_1\) is the \(1\times 1\) principal minor, we obtain
\[
  M=C_3C_6-C_4C_5+1.
\]

This identity can be used to compute the dimension of \(M^{\uii}[3,6]\).  
Since \(C_3^{\uii}\), \(C_4^{\uii}\), and \(C_5^{\uii}\) are fundamental modules with
\[
  \dim C_3^{\uii}=4,\quad
  \dim C_4^{\uii}=6,\quad
  \dim C_5^{\uii}=6.
\]
The short exact sequence
\[
  0\to L(Y_{3,8})
  \to L(Y_{2,9})\tens L(Y_{1,6})
  \to L(Y_{1,6}Y_{2,9})
  \to 0
\]
gives \(\dim C_6^{\uii}=6\cdot 4-4=20\).
Taking dimensions in the identity
\(M=C_3C_6-C_4C_5+1\), we get
\[
  \dim M
  =
  \dim C_3^{\uii}\dim C_6^{\uii}
  -
  \dim C_4^{\uii}\dim C_5^{\uii}
  +
  \dim \bone
  =
  4\cdot 20-6\cdot 6+1
  =
  45.
\]
We also remark that applying \(q\)-characters to the identity
\(M=C_3C_6-C_4C_5+1\) gives the \(q\)-character of \(M\) from the
\(q\)-characters of the modules \(C_t^{\uii}\).
\end{example}

\vskip 2em 

\providecommand{\bysame}{\leavevmode\hbox to3em{\hrulefill}\thinspace}
\providecommand{\MR}{\relax\ifhmode\unskip\space\fi MR }
\providecommand{\MRhref}[2]{%
	\href{http://www.ams.org/mathscinet-getitem?mr=#1}{#2}
}
\providecommand{\href}[2]{#2}

\end{document}